\documentclass[journal,draftcls,onecolumn,12pt,twoside]{IEEEtran}

\normalsize

\usepackage{cite}
\usepackage{amsmath,amssymb,amsfonts}
\usepackage{algorithmic}
\usepackage{graphicx}
\usepackage{textcomp}
\usepackage{xcolor}

\PassOptionsToPackage{T1}{fontenc}
\PassOptionsToPackage{utf8}{inputenc}
\PassOptionsToPackage{english}{babel}
\PassOptionsToPackage{svgnames}{xcolor}
\PassOptionsToPackage{ruled,boxed,linesnumbered,vlined,noend}{algorithm2e}
\PassOptionsToPackage{nameinlink,capitalize,noabbrev}{cleveref}
\PassOptionsToPackage{shortlabels}{enumitem}
\usepackage{%
  fontenc,inputenc,babel,%
  xstring,enumitem,algorithm2e,%
  amsmath,amssymb,amsthm,thm-restate,mathabx,mathtools,bm,%
  multirow,etoolbox,fancyhdr,lastpage,newfloat,tabularx,%
  todonotes,booktabs,%
  rotating,%
  tcolorbox,caption,%
  hyperref,cleveref,%
}

\SetKwInput{KwInput}{Input}
\SetKwInput{KwOutput}{Output}

\hypersetup{
  colorlinks,
  linkcolor=black,
  citecolor=black,
  urlcolor=black,
  bookmarksnumbered,
}
\newcommand{\doi}[1]{doi: \href{http://dx.doi.org/#1}{\path{#1}}}

\DontPrintSemicolon
\SetKw{break}{break}

\newtheorem{theorem}{Theorem}[section]
\newtheorem{proposition}[theorem]{Proposition}
\newtheorem{lemma}[theorem]{Lemma}
\newtheorem{corollary}[theorem]{Corollary}
\newtheorem{definition}[theorem]{Definition}
\newtheorem{remark}[theorem]{Remark}

\newtheorem{problem}[theorem]{Problem}
\numberwithin{equation}{section}

\newcommand{\ZZ}{\ensuremath{\mathbb{Z}}}

\newcommand{\QQ}{\ensuremath{\mathbb{Q}}}

\newcommand{\F}{\mathbb{F}}

\newcommand{\ffield}{F}
\newcommand{\vals}{\mathcal{F}}
\newcommand{\pts}{\mathcal{P}}
\newcommand{\X}{\mathcal{X}}
\newcommand{\Y}{\mathcal{Y}}
\newcommand{\tree}{\mathcal{T}}

\newcommand{\Hpol}{H}
\newcommand{\funcf}{F}
\newcommand{\Oh}{\mathcal{O}}
\newcommand{\sOh}{\tilde{\Oh}}

\newcommand{\Mult}[1]{\mathsf{M}(#1)}
\newcommand{\Multlog}[1]{\Mult{#1}\log({#1})}
\newcommand{\setcomp}[2]{\{ {#1} \enspace | \enspace {#2} \}}
\newcommand{\set}[1]{\{ {#1} \}}

\newcommand{\funcimpl}[2]{{#1} \longmapsto \enspace {#2}}

\newcommand{\ideal}[1]{\langle {#1} \rangle}
\newcommand{\pder}[2]{\frac{\partial {#1}}{\partial {#2}}}
\renewcommand{\vec}[1]{\boldsymbol{#1}}

\newcommand{\riroch}{\mathcal{L}}
\newcommand{\agcode}{\mathcal{C}}

\newcommand{\for}{\text{ for }}

\newcommand{\assign}{\leftarrow}

\newcommand{\Cset}{{\mathcal S}}

\newcommand{\quot}[1]{``#1''}

\newcommand{\dX }{d_X}
\newcommand{\dY }{d_Y}
\newcommand{\nX }{n_X}
\newcommand{\nY}{\nu_Y}

\newcommand{\UniMPE}{\ensuremath{\mathsf{UnivariateMPE}}\xspace}
\newcommand{\BiMPE}{\ensuremath{\mathsf{BivariateMPE}}\xspace}
\newcommand{\Enc}{\ensuremath{\mathsf{Encode}}\xspace}
\newcommand{\UnEnc}{\ensuremath{\mathsf{Unencode}}\xspace}
\newcommand{\UniInterp}{\ensuremath{\mathsf{UnivariateInterp}}\xspace}
\newcommand{\BiInterp}{\ensuremath{\mathsf{BivariateInterp}}\xspace}
\newcommand{\LinMod}{\ensuremath{\mathsf{Combine}}\xspace}

\newcommand{\Reduce}{\ensuremath{\mathsf{Reduce}}\xspace}
\newcommand{\TreeVanish}{\ensuremath{\mathsf{TreeVanish}}\xspace}
\newcommand{\VanTree}{\mathcal{U}}

\newcommand{\Grobner}{Gr\"obner\xspace}
\DeclareMathOperator{\LM}{\mathsf{LM}}
\DeclareMathOperator{\orderab}{\preceq_{\mathit{a,b}}}

\newcommand{\VanIdeal}{\mathcal{G}}

\DeclareMathOperator{\supp}{supp}
\DeclareMathOperator{\st}{\enspace s.t. \enspace}
\DeclareMathOperator{\ev}{ev}

\DeclareMathOperator{\rem}{rem}

\DeclarePairedDelimiter\abs{\lvert}{\rvert}

\newcommand{\Cab}{$C_{ab}$\xspace}
\newcommand{\HW}{\mathrm{HW}}

\begin{document}

\title{Fast Encoding of AG Codes over \Cab Curves
}

\author{Peter~Beelen,
        Johan~Rosenkilde,
        and~Grigory~Solomatov.%
      }



\maketitle

\begin{abstract}
  We investigate algorithms for encoding of one-point algebraic geometry (AG) codes over certain plane curves called \Cab curves, as well as algorithms for inverting the encoding map, which we call ``unencoding''.
  Some \Cab curves have many points or are even maximal, e.g.~the Hermitian curve.
  Our encoding resp.~unencoding algorithms have complexity $\sOh(n^{3/2})$ resp.~$\sOh(qn)$ for AG codes over any \Cab curve satisfying very mild assumptions, where $n$ is the code length and $q$ the base field size, and $\sOh$ ignores constants and logarithmic factors in the estimate.
  For codes over curves whose evaluation points lie on a grid-like structure, for example the Hermitian curve and norm-trace curves, we show that our algorithms have quasi-linear time complexity $\sOh(n)$ for both operations.
  For infinite families of curves whose number of points is a constant factor away from the Hasse--Weil bound, our encoding and unencoding algorithms have complexities $\sOh(n^{5/4})$ and $\sOh(n^{3/2})$ respectively.
\end{abstract}

\begin{IEEEkeywords}
  Encoding, AG code, Hermitian code, \Cab code, norm-trace curve
\end{IEEEkeywords}

\section{Introduction}
\subsection{Algebraic geometry codes}
In the following $\F$ is any finite field, while $\F_q$ denotes the finite field with $q$ elements.
An $\F$-linear $[n,k]$ code is a $k$-dimensional subspace $\mathcal C \subseteq \F^n$.
A substantial part of the literature on codes deals with constructing codes with special properties, in particular high minimum distance in the Hamming metric.
In this context, algebraic geometry (AG) codes, introduced by Goppa \cite{goppa_algebraico-geometric_1983}, have been very fruitful: indeed, we know constructive families of codes from towers of function fields whose minimum distance beat the Gilbert--Varshamov bound \cite{tsfasman_modular_1982}.
Roughly speaking, such codes arise by evaluating functions in points lying on a fixed algebraic curve defined over $\F$.
The evaluation points should be rational, i.e., defined over $\F$.

The well-known Reed--Solomon (RS) codes constitute a particularly simple subfamily of AG codes. Arguably, after Reed--Solomon (RS) codes, the most famous class of AG codes is constructed using the Hermitian curve; the Hermitian curve is a maximal curve, i.e.~the number of rational points meets the Hasse--Weil bound \cite[Theorem 5.2.3]{stichtenoth_algebraic_2009}.
It is an example of the much larger family of \Cab curves, which are plane curves given by a bivariate polynomial equation $H(X,Y) \in \F[X,Y]$ with several additional regularity properties.
These imply that the function field associated to a \Cab curve has a single place at infinity, $P_\infty$, and any function with poles only at $P_{\infty}$ can be represented by a bivariate polynomial $f \in \F[x,y]$ whose degree is bounded by a function of the pole order at $P_\infty$.
Here, $x$ and $y$ are two functions which satisfy $H(x,y) = 0$ and hence $\F[x,y]$ is a quotient of $\F[X,Y]$.

This means that the computations needed for operating with one-point AG codes over \Cab curves are much simpler than in the general AG code case.
The most well-studied operation pertaining to codes is decoding, i.e.~obtaining a codeword from a noisy received word.
For general AG codes, the fastest decoding algorithms essentially revert to linear algebra and have complexity roughly $\Oh(n^3)$, where $n$ is the length of the code, e.g.~\cite{sakata_generalized_1995,lee_unique_2014}.
However, for one-point \Cab codes, we have much faster algorithms, e.g.~\cite{beelen_efficient_2010}.
In \cite{nielsen_sub-quadratic_2015}, we studied Hermitian codes, i.e.~AG codes over the Hermitian curve and obtained a decoding algorithm with complexity roughly $\sOh(n^{5/3})$\footnote{
  Formally, for a function $f(n)$, then  $\sOh\big(f(n)\big) = \bigcup_{c=0}^\infty \Oh\big(f(n) \log^c(f(n)) \big)$.
}.

\subsection{Encoding and unencoding}
A simpler, though somewhat less studied problem for AG codes, however, is the \emph{encoding}, i.e.~the computational task of obtaining a codeword $\vec c \in \mathcal C$ belonging to a given message $\vec m \in \F^k$.
Given a message $\vec m$ and a generator matrix $G \in \F^{k \times n}$ of the code, a natural encoder is obtained as the vector-matrix product $\vec c = \vec mG$.
In general, this costs roughly $2kn$ operations in the field $\F$.
If $G$ is ``systematic'', e.g.~in row-reduced echelon form, then it is slightly cheaper, costing only $2k(n-k)$ operations.
Keeping the rate $k/n$ fixed and letting $n \rightarrow \infty$, the asymptotic cost is in both cases $\Oh(n^2)$ operations in $\F$.
For an arbitrary linear code, there is not much hope that we should be able to do better.

The inverse process of encoding, which we will call \emph{unencoding}, matches a given codeword $\vec c$ with the sent message $\vec m$.
If the encoder was systematic, this is of course trivial, but for an arbitrary linear encoder computing this inverse requires finding an information set for the code and inverting the generator matrix at those columns.
This matrix inverse can be precomputed, in which case the unencoding itself is simply a $k \times k$ vector-matrix multiplication costing $\Oh(k^2)$, which for a fixed rate equals $\Oh(n^2)$.

In this article, we contribute to the study of fast encoding and unencoding by investigating the case of one-point AG codes over any \Cab curve.
For such codes, encoding can be considered as follows: the entries of the message $\vec m \in \F^k$ are written as the coefficients of a bivariate polynomial $f_{\vec m} \in \F[X,Y]$ with bounded degree, and the codeword is then obtained by evaluating $f_{\vec m}$ at rational points of the \Cab curve (in some specific order).
This is called a `` multipoint evaluation'' of $f_{\vec m}$.
Similarly, for unencoding we are given a codeword $\vec c \in \F^n$ and we seek the unique polynomial $f \in \F[X,Y]$ whose monomial support satisfies certain constraints and such that the entries of $\vec c$ are the evaluations of $f$ at the chosen rational points of the \Cab curve.
This is called ``polynomial interpolation'' of the entries of $\vec c$.
Using this approach, we give fast algorithms for encoding and unencoding AG codes over any \Cab curve; in particular we obtain quasi-linear complexity $\sOh(n)$ in the code length for encoding and unencoding one-point Hermitian codes.

Our outset is to find algorithms for multipoint evaluation and interpolation of bivariate polynomials on any point set $\pts$, where we at first do not use the fact that $\pts$ are rational points on a \Cab curve; we do this in \cref{ssec:mult-eval-bivar,ssec:fast-interp} respectively.
Under mild assumptions, our algorithms for these problems have quasi-linear complexity in the input size when $\pts$ is a ``semi-grid'', i.e.~if we let $\Y_\alpha = \{ \beta \in \F \mid (\alpha, \beta) \in \pts \}$ for $\alpha \in \F$, then each $|\Y_\alpha|$ is either 0 or equals some constant $\nY$ independent of $\alpha$, see \cref{fig:semigrid} and \cref{def:semi-grid}.
This result may be of independent interest.
We then apply these algorithms to the coding setting in \cref{ssec:fast-encoding,ssec:unencoding} respectively.
In \cref{sec:applications} we more specifically study the performance for \Cab curves with special structure or sufficiently many points.

\subsubsection*{Contributions}
\begin{itemize}
  \item We give quasi-linear time algorithms for bivariate multipoint evaluation and interpolation when the point set is a semi-grid, under some simple conditions of the monomial support.
  See \cref{rem:mpe_semi-grid,rem:interp_semi-grid}.
  \item We give algorithms for encoding and for unencoding a one-point AG code over an arbitrary \Cab curve.
  Under very mild assumptions on the \Cab curve, these algorithms have complexity $\sOh(an) \subset \sOh(n^{3/2})$, respectively $\sOh(qn)$, where $a$ is the smallest nonzero element of the Weierstrass semigroup at $P_{\infty}$.
  Note $q < n$ for \Cab curves of interest to coding theory.
  Our interpolation algorithm requires a polynomial amount of precomputation time.
  The encoding is not systematic.
  See \cref{thm:Enc,thm:unencoding}.
  \item We show that for codes whose evaluation points are semi-grids in a particular ``maximal'' way compared to the \Cab curve, both algorithms have quasi-linear complexity $\sOh(n)$.
  This includes codes over the Hermitian curve and norm-trace curves.
  See \cref{def:cab_semigrid,prop:cab_semigrid_enc,prop:cab_semigrid_unenc,cor:hermit-cost,cor:norm_trace-cost,cor:hermitlike-cost}.
  \item We show that the algorithms have improved complexity if the \Cab curve has sufficiently many rational points.
  For example, if the number of rational points is a constant fraction from the Hasse--Weil bound, then the encoding algorithm has complexity $\sOh(n^{5/4})$, and the unencoding algorithm has complexity $\sOh(n^{3/2})$, see \cref{thm:hasse-weil-enc}.
\end{itemize}


\subsection{Related work on encoding}
For the particularly simple case of RS codes, it is classical that the encoding can be done in quasi-linear complexity \cite{justesen_complexity_1976} using univariate multipoint evaluation (also see \cref{ssec:comp-tools}).

Certain AG codes are investigated in \cite{heegard_systematic_1995}, where they give a space-efficient encoding algorithm (i.e.~it does not need to store a $k \times n$ generator matrix) using \Grobner bases and high-order automorphisms of the code; the time complexity, however, still remains quadratic.

In \cite{yaghoobian_1992} one-point Hermitian codes over $\F_{q^2}$ are encoded by viewing them as concatenated RS codes.
More precisely, if $f \in L(m P_\infty)$, then $f=\sum_{i=0}^\kappa f_i(x)y^i$ for suitable polynomials $f_i(x)$, where $\kappa = \min(q-1, \lfloor m/(q+1) \rfloor)$.
Evaluation of the $f_i(x)$ corresponds to fast encoding of RS codes.
Since for each $x$-coordinate in $\F_{q^2}$ there are exactly $q$ points on the Hermitian curve with this $x$-coordinate, the encoding of the $f_i(x)$ gives rise to a $q$-fold concatenation of an RS codeword.
Then the evaluation of $f_i(x)y^i$ is computed by multiplying each coordinate of the concatenated RS codeword with a suitable value.
They perform complexity analysis, but it seems that using fast RS encoding the algorithm costs $\sOh(\kappa q q^2) \subset \sOh(q^4)=\sOh(n^{4/3})$ operations in $\F_{q^2}$.
Though the underlying principle of our algorithm has similarities with this approach, our algorithm is a factor $n^{1/3}$ faster for the Hermitian curve.

In \cite{matsumoto_2001} the results of \cite{yaghoobian_1992} are generalized to arbitrary one-point AG codes.
Unfortunately, no asymptotic complexity analysis is given, making it difficult to compare their results with ours in terms of the parameters $n, k$ etc.
Their abstract does state that there are examples where their method is three times as fast as the trivial quadratic encoding, and so one may suspect that they have no asymptotic gain in the general case.
It can also be shown that their algorithm is never quasi-linear, which ours is in certain cases.


In \cite{narayanan_nearly_2017} an encoding algorithm is given which is faster than $\Oh(n^{3/2})$ for certain carefully tailored, asymptotically good sub-codes of AG codes arising from the Garcia-Stichtenoth tower \cite{garcia_tower_1995}.
Our methods do not handle these codes, so the results can not be directly compared.

\subsection{Related work on bivariate multipoint evaluation}
\label{sec:relat-work-bivar}

As outlined above, multipoint evaluation (MPE) and interpolation of bivariate polynomials over given point sets is a computational problem tightly related to encoding and unencoding of AG codes over \Cab curves.
In fact, any MPE algorithm for bivariate polynomials can immediately be applied for encoding.
The situation is somewhat more complicated for interpolation, which we get back to.
In this section and the next we review the literature on these problems\footnote{%
  Some of the cited algorithms apply to more variables than just 2, but we specialize the discussion to the bivariate case to ease comparison with our results.}.

We begin by discussing the MPE problem.
The input is a point set $\pts \subset \F^2$ and $f \in \F[X,Y]$ with $\deg_X f = d_X$ and $\deg_Y f = d_Y$, and we seek $\big(f(P)\big)_{P \in \pts}$.
Let $n := |\pts|$.
As we will see in \cref{ssec:cab_codes}, the main interest for the application of \Cab codes is when $d_X d_Y < n$ and $d_Y \ll d_X$.
The former is a common assumption in the literature, but numerous papers assume $d_X \approx d_Y$ and such algorithms will often have poor performance in our case.

Spurred by the quasi-linear algorithm available in the univariate case (see \cref{ssec:comp-tools}), the best we could hope for would be an algorithm of complexity $\sOh(d_X d_Y + n)$, i.e.~quasi-linear in the size of the input, but such a result is still not known in general.
We will exemplify the complexities here for use in encoding Hermitian codes, i.e.~\Cab codes over the Hermitian curve, see \cref{ssec:hermitian}: in this case $n = q^3$ where we work over the field $\F_{q^2}$.
We will consider the case where the dimension of the code is in the order of the length, for which we then have $\deg_X f \in \Oh(n^{2/3})$ and $\deg_Y f < q = n^{1/3}$. The naive approach is to compute the evaluations of $f$ one-by-one.
Using Horner's rule, each such evaluation can be computed in $\Oh(d_Xd_Y)$ time, for a total complexity of $\Oh(d_Xd_Yn)$.
For the Hermitian codes, the complexity would be $\Oh(n^2)$.

One of the first successes was obtained by Pan \cite{pan_simple_1994} with a quasi-linear algorithm for the case $\pts = S_X \times S_Y$ for $S_X, S_Y \subseteq \F$, i.e.~evaluation on a grid, see \cref{fig:grid}.
The algorithm works by applying univariate MPE in a ``tensored'' form.
The algorithm can be directly applied for any $\pts$ by calling it on the smallest grid $\hat \pts$ which contains all of $\pts$ and then throwing away the unneeded evaluations.
If $|\hat \pts| \gg n$ then the complexity of this approach will not be quasi-linear in the original input size: in the worst case $|\hat \pts| \approx n^2$ so the complexity becomes $\sOh(d_Xd_Y + n^2)$, which is quadratic in the input size when $d_Xd_Y < n$.
For Hermitian codes, then $\hat \pts = \F_{q^2}^2$, hence Pan's algorithm would give complexity $\sOh(n^{4/3})$.
Our MPE algorithm presented in \cref{ssec:mult-eval-bivar} essentially generalizes Pan's algorithm to achieve quasi-linear cost on point sets with only \emph{semi-grid} structure, see \cref{fig:semigrid}; therefore the performance of our algorithm is never worse than Pan's.
Moreover, though few \Cab curves form a grid, we observe in \cref{ssec:fast-unencoding} that certain nice families, including the Hermitian curve, form semi-grids, implying that our evaluation algorithm has quasi-linear complexity for codes over these curves.

\begin{figure}
\centering
\begin{minipage}{.5\textwidth}
  \centering
  \includegraphics[width=.7\linewidth]{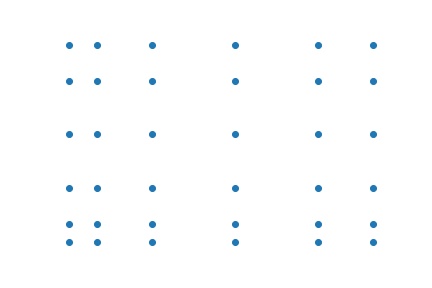}
  \captionof{figure}{A grid.}
  \label{fig:grid}
\end{minipage}%
\begin{minipage}{.5\textwidth}
  \centering
  \includegraphics[width=.7\linewidth]{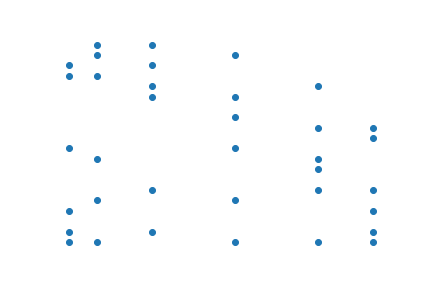}
  \captionof{figure}{A semi-grid.}
  \label{fig:semigrid}
\end{minipage}
\end{figure}

N\"usken and Ziegler \cite{nusken-ziegler-2004} (NZ) reduced bivariate multipoint evaluation to a variant of \emph{bivariate modular composition}:
Write $\pts = \{ (\alpha_1,\beta_1),\ldots,(\alpha_n,\beta_n) \}$ and assume all the $\alpha_i$ are distinct.
Compute $h(X) = \prod_{i=1}^n (X - \alpha_i)$ and $g \in \F[X]$ such that $g(\alpha_i) = \beta_i$ for $i = 1,\hdots,n$; both of these can be computed in $\sOh(n)$ time.
If we then compute $\rho(x) = f(x,g(x)) \rem h(x) \in \F[X]$, we see that $\rho(\alpha_i) = f(\alpha_i, \beta_i)$ for each $i$, and hence we can compute the evaluations of $f$ at the points $\pts$ by a univariate MPE of $\rho$ at the $\alpha_1,\ldots,\alpha_n$.
The latter can be done in $\sOh(n)$ time, so all that remains is the computation of $\rho(X)$.
N\"usken and Ziegler show how to do this in complexity roughly $\Oh(d_X d_Y^{1.635 + \epsilon} + nd_Y^{0.635+\epsilon})$, where $\epsilon > 0$ can be chosen arbitrarily small, using fast rectangular matrix multiplication \cite{le_gall_rect_mat_2012}.

In general, and in the case of our interest, the $X$-coordinates of $\pts$ will not be distinct.
In this case, the points can be \quot{rotated} by going to an extension field $K$ with $[K:\F] = 2$: choose any $\theta \in K \setminus \F$, and apply the map $\funcimpl{(\alpha,\beta)}{(\alpha + \theta \beta,\beta)}$ to the points $\pts$, and replace $f$ by $\hat f := f(X-\theta Y,Y)$.
We can now apply the NZ algorithm; that the operations will take place in $K$ costs only a small constant factor compared to operations in $\F$ since the extension degree is only $2$.
The main problem is that $\deg_Y \hat f$ is now generically $\max(d_X, d_Y)$, so assuming $d_Y < d_X$, the complexity of the NZ algorithm becomes roughly $\Oh((n + d_X^2)d_X^{0.635 + \epsilon})$.
For Hermitian codes, this yields $\Oh(n^{1.756 + \epsilon})$.

In their celebrated paper \cite{kedlaya-umans-2008}, Kedlaya and Umans (KU) gave an algorithm for bivariate MPE with complexity $\Oh((n + d_X^2)^{1 + \epsilon})$ bit operations for any $\epsilon > 0$, assuming $d_Y < d_X$.
In outline, the algorithm works over prime fields by lifting the data to integers, then performing the MPE many times modulo many small primes, and then reassembling the result using the Chinese Remainder theorem.
Over extension fields, some more steps are added for the lift to work.
Note that the KU algorithm has quasi-linear complexity when $d_Y \approx d_X$.
As mentioned, our main interest is $d_Y \ll d_X$.
For our running example of the Hermitian code, applying the KU algorithm has complexity $\Oh(n^{4/3 + \epsilon})$.

\subsection{Related work on bivariate interpolation}
Let us turn to the interpolation problem.
The input is now a point set $\pts \subset \F^2$ and interpolation values $\vals: \pts \rightarrow \F$, and we seek $f \in \F[X,Y]$ such that $f(P) = \vals(P)$ for each $P \in \pts$.
There are infinitely many such $f$, so to further restrict, or even make the output unique, one has to pose restrictions on the monomial support on the output $f$.
We discuss the setting relevant to us in \cref{ssec:encoding-map}.

Efficient interpolation algorithms include Pan's \cite{pan_simple_1994}, which works for points on grids, and its generalization by van der Hoeven and Schost \cite{van_der_hoeven_multi-point_2013}, which works for certain structured subsets of grids.
The monomial support output by these algorithms does not match our requirements.
Given a $\hat f \in \F[X,Y]$ which correctly interpolates the sought values, but has incorrect monomial support,
a general way to solve the problem is to let $\VanIdeal \subset \F[X,Y]$ be the ideal of all polynomials which vanish at the points of $\pts$, and then precompute a \Grobner basis $G$ of $\VanIdeal$ under an appropriate monomial order: then $\hat f \rem G$, the unique remainder of $\hat f$ divided by the basis $G$, will have ``minimal'' possible monomial support under this order.
One may use van der Hoeven's fast division algorithm \cite{van_der_hoeven_complexity_2015} for this step.
This is exactly the strategy we use in \cref{sec:find-small-interp}, where we first compute an $\hat f$ by generalizing Pan's interpolation algorithm to work for semi-grids.
This $\hat f$ can be described by a closed-form expression, see \cref{lem:explicit_interpolation}.
A similar expression was used in a decoding algorithm for AG codes over the Hermitian curve in \cite{lee_list_2009}, and it was shown in \cite{nielsen_sub-quadratic_2015} how to compute it fast; that approach can be seen as a special case of our algorithm.

A very different, and very flexible, interpolation algorithm is simply to solve the interpolation constraints as a system of linear constraints in the coefficients to the monomials in the monomial support.
Solving the resulting $n \times n$ linear system using Gaussian elimination would cost $\Oh(n^\omega)$, where $\omega < 2.37286$ is the exponent of matrix multiplication \cite{le_gall_powers_2014}.
In our case, we can do much better by observing that for the monomial support we require, the linear system would have low \emph{displacement rank}, namely $a$, so we could use the algorithm for structured system solving by Bostan et al.~\cite{bostan_solving_2008,bostan_matrices_2017} for a cost of $\sOh(a^{\omega-1} n)$.
For the Hermitian codes this yields a complexity of roughly $\sOh(n^{1.458})$, which is much slower than ours.
However, for general \Cab codes this is our main contender, and we compare again in \cref{ssec:unencoding,ssec:encoding-good-curves}: the take-away is that for most parameters of interest, our algorithm seems to be faster.

\section{Preliminaries}
\label{sec:preliminaries}

\subsection{Codes from \Cab curves}
\label{ssec:cab_codes}

In this subsection we discuss in some detail the family of AG codes for which we want to find fast encoders and unencoders. Note that these AG codes and the algebraic curves used to construct them were previously studied in \cite{miura-1993,miura-kamiya-1993}; these results will be mentioned here.
Also they occur as a special case of the codes and curves studied in \cite{feng_rao_1994,handbook}.

For a bivariate polynomial $\Hpol = \sum_{i,j} a_{ij}X^iY^j \in \F[X,Y]$ with coefficients in a finite field $\F$, we define $\supp(\Hpol)=\{X^iY^j \mid a_{ij} \neq 0\}.$

 \begin{definition}\label{def_Cabpol}
Let $a,b$ be positive, coprime integers. We say that a bivariate polynomial $\Hpol \in \F[X,Y]$ is a \Cab polynomial if:
 \begin{itemize}
 \item $X^b,Y^a \in \supp(\Hpol)$,
 \item $X^iY^j \in \supp(\Hpol) \implies ai + bj \leq ab$,
 \item The ideal $\ideal{\Hpol,\pder{\Hpol}{X},\pder{\Hpol}{Y}} \subseteq \F[X,Y]$ is equal to the unit ideal $\F[X,Y]$.
 \end{itemize}
\end{definition}

\begin{remark}
  In our algorithms the two variables $X$ and $Y$ are treated differently, which entails that complexities are not invariant under swapping of $X$ and $Y$ in $\Hpol$. We will commit to the arbitrary choice of ``orienting'' our algorithms such that their complexities depend explicitly only on $a$, which means that whenever the input is not assumed to have special structure which depends on $a$ and $b$, it is of course sensible to swap $X$ and $Y$ so $a \leq b$. Note that the case $a=b$ can only occur when $a=b=1$ since they are coprime. We will disregard this degenerate case.
\end{remark}

 Define $\deg_{a,b}$ to be the $(a,b)$-weighted degree of a bivariate polynomial.  More concretely: $\deg_{a,b}(X^i Y^j) = ai + bj$. The first two conditions imply that $\Hpol(X,Y)=\alpha X^b+\beta Y^a+G(X,Y),$ where $\alpha,\beta \in \F \setminus \{0\}$ and $\deg_{a,b}(G(X,Y)) < ab.$
In particular, the polynomial $\Hpol$ is absolutely irreducible \cite[Cor. 3.18]{handbook}. The theory of Newton polygons, i.e., the convex hull of $\{(i,j) \mid X^iY^j \in \supp(\Hpol)\}$, can also be used to conclude this \cite{gao-2001}.

This implies that the a \Cab polynomial defines an algebraic curve. Following \cite{miura-1993}, the type of algebraic curves obtained in this way are called \Cab curves. As observed there, these curves, when viewed as projective curves, have exactly one point at infinity $P_\infty$, which, if singular, is a cusp. What this means can be explained in a very simple way using the language of function fields. A given \Cab polynomial $\Hpol$ defines a \Cab curve, with function field $\funcf = \F(x,y)$ obtained by extending the rational function field $\F(x)$ with a variable $y$ satisfying $\Hpol(x,y) = 0$. Since $\Hpol$ is absolutely irreducible, $\F$ is the full constant field of $\funcf.$ The statement that the point $P_\infty$, if it is a singularity, is a cusp, just means that the function $x$ has exactly one place of $\funcf$ as a pole. With slight abuse of notation, we denote this place by $P_\infty$ as well. The defining equation of a \Cab curve, directly implies that for any $i,j \in \mathbb{Z}$, the function $x^iy^j$ has pole order $\deg_{a,b}(x^iy^j)=ai+bj$ at $P_\infty.$ In particular, $x$ has pole order $a$ and $y$ has pole order $b$ at $P_\infty.$

The \emph{genus} of a function field is important for applications in coding theory, since it occurs in the Goppa bound on the minimum distance of AG codes. It is observed in \cite{miura-1993} that the genus of the function field $\funcf= \F(x,y)$ defined above equals $g=(a-1)(b-1)/2$.  Indeed, this is implied by the third condition in \cref{def_Cabpol}, also see \cite[Theorem 4.2]{beelen-pellikan-2000}. We collect some facts in the following proposition. These results are contained in \cite{miura-1993}, expressed there in the language of algebraic curves.

\begin{proposition}[\protect{\cite{miura-1993}}]
  \label{prop:c-ab}
Let $\Hpol \in \F[X,Y]$ be a \Cab polynomial and $\funcf=\F(x,y)$ the corresponding function field. Then $\funcf$ has genus $g =\frac{1}{2}(a-1)(b-1).$ The place $P_\infty$ is rational and a common pole of the functions $x$ and $y$ and in fact the only place which is a pole of either $x$ or $y$. For any $i,j \in \mathbb{Z}$, the function $x^iy^j \in \funcf$ has pole order $\deg_{a,b}(x^iy^j)=ai+bj$ at $P_\infty.$
\end{proposition}

For a divisor $D$ of the function field $\funcf$, we denote by $\riroch(D)$ the Riemann--Roch space associated to $D$.
The third condition in \cref{def_Cabpol} implies that a \Cab curve cannot have singularities, apart from the possibly singular point at infinity. This has two important consequences. In the first place, all rational places of $\funcf$ distinct from $P_\infty$, can be identified with the points $(\alpha,\beta) \in \F^2$ satisfying $\Hpol(\alpha,\beta)=0$. We will call these places the finite rational places of $\funcf$.
Throughout the paper we will, by a slight abuse of notation, use a finite place $P_{\alpha,\beta}$ and its corresponding rational point $(\alpha,\beta)$ interchangeably.
A second consequence, as observed in \cite{miura-1993}, is that the functions in $\F[x,y]$ are the \emph{only} functions in $\funcf$ with poles only at $P_\infty$; in other words, $\riroch(mP_\infty) = \{ f \in \F[x,y] \mid \deg_{a, b}(f) \leq m \}$.
Since $\Hpol(x,y)=0$, any $f \in \F[x,y]$ can be uniquely written as a polynomial with $y$-degree at most $a-1$.
We will call this the \emph{standard form} of $f$.
We are now ready to define \Cab codes.

\begin{definition}
  \label{def:cab-codes}
Let $\Hpol$ be a \Cab polynomial and $\funcf$ the corresponding function field. Further, let $P_1,\hdots,P_n$ be distinct, finite rational places of $\funcf$ and let $m$ be a non-negative integer. Then the \Cab code of order $m$ is defined to be:
\begin{align*}
  \agcode_{\Hpol}(\pts,m) &:= \setcomp{ ( \ev_\pts(f) }{f \in \riroch(mP_\infty)} \subseteq \F^n \ , \textrm{ where} \\
  \ev_\pts(f) &= \big(  f(P_1), \hdots, f(P_n) \big) \ .
\end{align*}
\end{definition}
In the standard notation for AG codes used for example in \cite{stichtenoth_algebraic_2009}, the code $\agcode_{\Hpol}(\pts,m)$ is equal to the code $C_{\mathcal L}(D,mP_\infty),$ with $D=P_1+\cdots+P_n.$
Since the divisor $mP_\infty$ is a multiple of a single place, the codes $\agcode_{\Hpol}(\pts,m)$ are examples of what are known as one-point AG codes.
Using for example \cite[Theorem 2.2.2]{stichtenoth_algebraic_2009}, we obtain that $\agcode_{\Hpol}(\pts,m)$ is an $[n, k, d]$ linear code, where $k = \dim\big( \riroch(mP_\infty)\big)-\dim\big( \riroch(mP_\infty-D)\big)$  and $d \geq n-m.$
In particular, $k=n$ if $m > n+2g-2$.
Therefore we will from now on always assume that $m \le n+2g-1.$
If $m<n,$ then $k = \dim\big( \riroch(mP_\infty)\big) \ge m+1-g$ and if additionally $2g-2 <m$, then $k=m+1-g.$
The precise minimum distance of a  \Cab code is in general not known from just the defining data.
Lastly, we also note the obvious bound $n \leq q^2$, where $q = |\F|$, due to the identification of rational places with points in $\F^2$.

When comparing algorithms pertaining to AG codes, as well as many other types of codes, it is customary to assume that the dimension $k$ grows proportional to $n$, denoted $k \in \Theta(n)$, i.e.~that the rate goes to some constant as $n \rightarrow \infty$.
For any such family of \Cab codes this implies that $m \in \Theta(n)$.
This in turn means that any message polynomial $f \in \riroch(mP_\infty)$ in standard form satisfies $\deg_Y f < a$ and $\deg_X f < m/a \in \Theta(n/a)$.
For most interesting \Cab-codes the genus $g \ll n$, so $ab \ll n$, while $n$ is relatively large compared to $q$, as measured e.g.~against the Hasse--Weil bound, see \cref{ssec:encoding-good-curves}. This means that in the cases of most interest to us, the message polynomials tend to have very different $X$ and $Y$ degrees.

\subsection{The evaluation-encoding map}
\label{ssec:encoding-map}

An encoding for a linear code such as $\agcode_\Hpol(\pts, m)$ is a linear, bijective map $\phi: \F^k \rightarrow \agcode_\Hpol(\pts, m) \subseteq \F^n$.
Computing the image of $\phi$ for some $\vec m \in \F^k$ is called ``encoding'' $\vec m$.
The process of computing the inverse, i.e.~given a codeword $\vec c \in \agcode_\Hpol(\pts, m)$ recover the message $\vec m := \phi^{-1}(\vec c)$, is often unnamed in the literature.
For lack of a better term (and since ``decoding'' is reserved for error-correction), we will call it ``unencoding''.

In light of \cref{def:cab-codes}, we can factor $\phi$ as $\phi = \ev_\pts \circ \varphi$, where $\varphi: \F^k \rightarrow \riroch(mP_\infty)$ is linear and injective.
If we choose $\varphi$ sufficiently simple and such that it outputs elements of $\F[x,y]$ in standard form, the computational task of applying $\phi$ reduces to computing $\ev_\pts$, i.e.~multipoint evaluation of bivariate polynomials of $(a,b)$-weighted degree at most $m$.
A natural basis for $\riroch(mP_\infty)$ is
\begin{equation}
  \label{eqn:msg_basis}
  B = \{ x^i y^j \mid \deg_{a,b}(x^iy^j) \leq m \land j \leq a-1 \} \ .
\end{equation}
If $k = \dim(\riroch(mP_\infty))$, then $|B| = k$, and we therefore choose $\varphi$ as taking the elements of a message $\vec m$ as the coefficients to the monomials of this basis in some specified order.
Then applying $\varphi$ takes no field operations at all.

If $k < \dim(\riroch(mP_\infty))$ then $|B| > k$, and this may happen when $m \geq n$.
We should then choose a subset $\hat B \subset B$ of $k$ monomials such that the vectors $\{ \ev_\pts(x^iy^j) \}_{x^i y^j \in \hat B}$ are linearly independent.
For our encoding algorithms, the choice of $\hat B$ will not matter.
However, for unencoding, we will assume that this choice has been made so that the monomials in $\hat B$ are, when sorted according to their $(a,b)$-weighted degrees, lexicographically minimal.
Put another way, a monomial $x^i y^j \in B$ is \emph{not} in $\hat B$ exactly when there is a polynomial $g \in \ker(\ev_\pts)$ with $\LM_{\orderab}(g) = x^i y^j$, where $\LM_{\orderab}$ denotes leading monomial according to $\orderab$, the $(a,b)$ weighted degree breaking ties by $x^b \orderab y^a$.
Such monomials $x^i y^j$ are what we will call ``reducible'' monomials in \cref{sec:find-small-interp}.

$\hat B$ is easy to precompute: start with $\hat B = \emptyset$, and go through the monomials of $B$ in order of increasing $(a,b)$-weighted degree.
For each such $x^u y^v$ if $\ev_\pts(x^uy^v)$ is linearly independent from $\{ \ev_\pts(x^iy^j) \}_{x^i y^j \in \hat B}$, then add to $\hat B$.

\subsection{Notation and computational tools}
\label{ssec:comp-tools}

For any point set $\pts \subseteq \F^2$ we define $\X(\pts) := \setcomp{\alpha \in \F}{\exists \beta \in \F \st (\alpha,\beta) \in \pts}$, i.e. the set of all $X$-coordinates that occur in $\pts$.
We write $\nX(\pts) := \abs{\X(\pts)}$ for the number of distinct $X$-coordinates.
Similarly, for any $\alpha \in \F$ we define $\Y_{\alpha}(\pts) := \setcomp{\beta \in \F}{(\alpha,\beta) \in \pts}$, i.e. the set of $Y$-coordinates that occur for a given $X$-coordinate $\alpha$,
and we let
\[
  \nY(\pts) := \max_{\alpha \in \pts}\abs{\Y_{\alpha}(\pts)} .
\]
In discussions where it is clear from the context which point set $\pts$ we are referring to, we may simply write $\X,\Y,\nX,\nY$.
Note that if $\pts$ is a subset of the rational points of a \Cab curve with polynomial $H(X,Y)$, then $\nY(\pts) \leq a =: \deg_Y(H)$ since for any given value of $\alpha$, there are at most $a$ solutions to the resulting equation in $H(\alpha, Y)$.

\begin{definition}
  \label{def:semi-grid}
  A point set $\pts \subset \F^2$ is a \emph{semi-grid} if $|\Y_\alpha(\pts)| \in \{ 0, \nY(\pts) \}$ for each $\alpha \in \F$.
\end{definition}

As outlined in \cref{ssec:cab_codes}, we distinguish between the bivariate polynomial ring $\F[X,Y]$ and the subset of functions in the \Cab function field spanned by $x$ and $y$, denoted $\F[x,y]$.
However, there is a natural inclusion map of functions $f \in \F[x,y]$ in standard form into a polynomial $f(X,Y) \in \F[X,Y]$ of $Y$-degree less than $a$.
In discussions and algorithms, we sometimes abuse notation by more or less explicitly making use of this inclusion map.

For ease of notation, our algorithms use \emph{lookup tables}, also known as dictionaries or associative arrays.
This is just a map $\mathcal{A} \rightarrow \mathcal{B}$ between a finite set $\mathcal{A}$ and a set $\mathcal{B}$ but where all the mappings have already been computed and stored, and hence can quickly be retrieved.
We use the notation $\vals \in \mathcal B^{\mathcal A}$ to mean a lookup table from $\mathcal A$ to $\mathcal B$.
For $a \in \mathcal A$, we write $\vals[a] \in \mathcal B$ for the mapped value stored in $\vals$.
Note that this is for notational convenience only: all our uses of lookup tables could be replaced by explicit indexing in memory arrays, and so we will assume that retrieving or inserting values in tables costs $O(1)$.

Our complexity analyses count basic arithmetic operations in the field $\F$ on an algebraic RAM model.
We denote by $\Mult n$ the cost of multiplying two univariate polynomials in $\F[X]$ of degree at most $n$.
We can take $\Mult{n} \in O(n \log n \log\log n)$ \cite{cantor_fast_1991}, or the slightly improved algorithm of \cite{harvey_faster_2017} with cost $\Mult{n} \in O(n \log n\ 8^{\log^* n})$, both of which are in $\sOh(n)$.
For precision, our theorems state complexities in big-$\Oh$ including all log-factors, and we then relax the expressions to soft-$\Oh$ for overview.

Our algorithms take advantage of two fundamental computational tools for univariate polynomials: \emph{fast multipoint evaluation} and \emph{fast interpolation}.
These are classical results, see e.g.~\cite[Corollaries 10.8 and 10.12]{von_zur_gathen_modern_2012}.

\begin{proposition}
  \label{prop:uni-mpe}
  There exists an algorithm $\UniMPE$ which inputs a univariate polynomial $h \in \F[Z]$ and evaluation points $\Cset \subseteq \F$, and outputs a table $\vals: \F^\Cset$ such that $\vals[\alpha] = h(\alpha)$ for every $\alpha \in \Cset$.
  It has complexity
  \[
    \Oh(\Multlog{\deg h + \abs{\Cset}}) \subset \sOh(\deg h + \abs{\Cset})
  \]
  operations in $\F$.
\end{proposition}

\begin{proposition}
  \label{prop:uni-interp}
  There exists an algorithm $\UniInterp_Z$ which inputs evaluation points $\Cset \subseteq \F$ and evaluation values $\vals \in \F^{\Cset}$, and outputs the unique $f \in \F[Z]$ such that $\deg f < k$ and $f(\alpha) = \vals[\alpha]$ for each $\alpha \in \Cset$, where $k = |\Cset|$.
  It has complexity $\Oh(\Multlog{k}) \subset \sOh(k)$ operations in $\F$.
\end{proposition}

\section{A fast encoding algorithm}
\label{sec:fast-encoding}

Let us now consider an algorithm for computing the encoding map for \Cab codes. We are given a message vector $f \in \riroch(mP_\infty) \subset \F[X,Y]/\ideal{\Hpol}$ and $n$ finite rational places $P_1, \hdots, P_n$ of $F$; we wish to compute $f(P_1),\hdots,f(P_n)$. We translate this problem into bivariate polynomial multipoint-evaluation by lifting $f$ to a polynomial in $\F[X,Y]$ in standard form, and identifying each $P_i$ with a pair $(\alpha_i,\beta_i), \in \F^2$ such that $\Hpol(\alpha_i,\beta_i) = 0$ for $i = 1, \hdots, n$. In the following subsection we will focus on the evaluation problem at hand, while in \cref{ssec:fast-encoding} we will apply the results to encoding of codes over \Cab curves.

\subsection{Multipoint-Evaluation of Bivariate Polynomials}
\label{ssec:mult-eval-bivar}
Let us for now forget that we originally came from the setting of codes. Suppose that we are given a set $\pts \subseteq \F^2$ of points with $\abs{\pts} = n$ and a bivariate polynomial $f \in \F[X,Y]$ with $\deg_Xf = d_X$ and $\deg_Yf = d_Y$. This naive approach will have complexity $\Oh(nd_Xd_Y)$.

We will generalize Pan's multipoint evaluation algorithm \cite{pan_simple_1994} (see \cref{sec:relat-work-bivar}), and show that it performs well on point sets $\pts$ where most $|\Y_\alpha(\pts)|$ are roughly the same size for each $\alpha \in \X(\pts)$.

The idea of the algorithm is the following:
we write
\begin{align*}
  f(X,Y) = f_0(X) + f_1(X)Y + \hdots + f_{d_Y}(X)Y^{d_Y} \quad, f_i \in \F[X] \ ,
\end{align*}
and then proceed by $d_Y+1$ univariate multipoint evaluations of the polynomials $f_i(X)$, $i = 0,\hdots,d_Y$, each evaluated on the values $\X(\pts)$.
For each $\alpha \in \X(\pts)$, we can therefore construct a univariate polynomial in $Y$ without further computations:
\[
  g_\alpha(Y) = f_0(\alpha) + f_1(\alpha)Y + \hdots + f_{d_Y}(\alpha)Y^{d_Y} \quad, f_i \in \F[X] \ .
\]
Again using univariate multipoint evaluation, we obtain $g_\alpha(\beta) = f(\alpha, \beta)$ for each $\beta \in \Y_\alpha(\pts)$. For the algorithm listing see \cref{algo:bi-eval}.

\begin{algorithm}
  \caption{$\BiMPE$: Bivariate multipoint evaluation}\label{algo:bi-eval}
  \KwInput{\;
    Bivariate polynomial $f = f_0(X) + f_1(X)Y + \hdots + f_{d_Y}(X)Y^{d_Y} \in \F[X,Y]$. \;
    Evaluation points $\pts \subseteq \F^2$.
}    
\KwOutput{\;
  Evaluation values $\vals = (f(\alpha,\beta))_{(\alpha,\beta) \in \pts} \in \F^{\pts}$.\;}
$\X \assign \X(\pts)$
\label{line:bi-eval:1}\;
$\Y_{\alpha} \assign \Y_{\alpha}(\pts)$
\label{line:bi-eval:2}\;
\lFor{$i = 1,\hdots,d_Y$}{$\vals_i \assign \UniMPE(f_i, \X) \in \F^{\X}$}
\label{line:bi-eval:3}
\ForEach{$\alpha \in \X$}{
  $g_{\alpha} \assign \sum_{i=0}^{d_Y} \vals_i[\alpha] Y^i \in \F[Y]$ \label{line:bi-eval:4}\;
  $\mathcal{G}_{\alpha} \assign \UniMPE(g_{\alpha},\Y_{\alpha}) \in \F^{\Y_{\alpha}}$ \label{line:bi-eval:5}\;
}
\Return $\vals \assign (\mathcal{G}_{\alpha}[\beta])_{(\alpha,\beta) \in \pts} \in \F^{\pts}$.
\label{line:bi-eval:6}

\end{algorithm}
\begin{theorem} \label{thm:bi-eval}
  \cref{algo:bi-eval} is correct.
  It has complexity
  \begin{align*}
    \Oh(d_Y\Multlog{d_X + \nX } + \nX \Multlog{d_Y + \nY}) \subset \sOh(d_Yd_X + \nX (d_Y + \nY))
  \end{align*}
  operations in $\F$, where $d_X = \deg_Xf$, $\nX = \nX(\pts)$, and $\nY = \nY(\pts)$.
\end{theorem}
\begin{proof}
 Correctness follows from the fact that
\begin{align*}
\vals[\alpha,\beta] = \mathcal{G}_{\alpha}[\beta] = g_{\alpha}(\beta) = \sum_{i=0}^{d_Y} \vals_{i}[\alpha] \beta^i = \sum_{i=0}^{d_Y} f_i(\alpha) \beta^i = f(\alpha,\beta) \ .
\end{align*}

For the complexity, \cref{line:bi-eval:1,line:bi-eval:2} both have cost $\Oh(n)$.
\cref{prop:uni-mpe} implies that computing $\vals_i$ costs
\begin{align*}
\Oh(\Multlog{\deg f_i + \nX }) \for i = 1,\hdots,d_Y \ ,
\end{align*}
thus the total cost for \cref{line:bi-eval:3} becomes
\begin{align*}
\Oh(d_Y\Multlog{d_X + \nX }).
\end{align*}
\cref{line:bi-eval:4} costs no operations in $\F$.
From \cref{prop:uni-mpe} it follows that computing each $\mathcal{G}_{\alpha}$ costs
\begin{align*}
  \Oh(\Multlog{d_Y + \abs{\Y_{\alpha}}}) \for \alpha \in \X,
\end{align*}
thus the total cost of \cref{line:bi-eval:5} becomes
\begin{align*}
\Oh(\nX \Multlog{d_Y + \nY}) \ .
\end{align*}
\cref{line:bi-eval:6} costs no operations in $\F$, and so the total cost of computing $\vals$ becomes as in the theorem.
\end{proof}

\begin{remark}
  \label{rem:mpe_semi-grid}
  If $\pts \in \F^2$ is a semi-grid, and $f \in \F[X,Y]$ is a dense polynomial satisfying either $n_X \in \Oh(\deg_X f)$ or $\deg_Y(f) \in \Oh(\nY)$, then \cref{algo:bi-eval} has quasi-linear complexity in the input size $|\pts| + \deg_X f \deg_Y f$.
\end{remark}

\begin{remark}
  Even if we use classical polynomial multiplication, with $\Mult n  = \Oh(n^2)$, and if we assume $d_X \in \Theta(n_X)$ and $d_Y \in \Theta(\nY)$, then the cost of \cref{algo:bi-eval} is $\sOh(\nY^2 n_X + \nY n_X^2)$ which for most point sets is better than the naive approach of point-by-point evaluation costing $\Oh(\nY n_X n)$.
\end{remark}

\subsection{Fast encoding}
\label{ssec:fast-encoding}

\cref{algo:bi-eval} gives rise a fast encoder; details can be found in \cref{algo:Enc}.

\begin{algorithm}
  \caption{\Enc}\label{algo:Enc}
  \KwInput{
    A \Cab code $\agcode_\Hpol(\pts,m) \subseteq \F^n$ of dimension $k$, with $\pts = \{ P_1, \hdots, P_n \}$ being finite rational places.
    Message $\vec m \in \F^k$.
  }
  \KwOutput{Codeword $\vec{c} = \phi(\vec m)$, where $\phi: \F^k \rightarrow \agcode_{\Hpol}(\pts,m)$ is the encoding map defined in \cref{ssec:encoding-map}.}
  $f \assign \varphi(\vec m) \subset \F[x,y]$ in standard form, where $\varphi$ is as in \cref{ssec:encoding-map} \;
  $\vals \assign \BiMPE(f, \pts) \in \F^{\pts}$, where $f$ is lifted to $\F[X,Y]$ \label{line:enc:1} \;
  \Return $(\vals[P_1], \hdots, \vals[P_n]) \in \F^n$
\end{algorithm}

\begin{theorem} \label{thm:Enc}
  \cref{algo:Enc} is correct.
  It uses at most $\Oh(\Multlog{m + a\nX }) \subset \sOh(m + a\nX)$ operations in $\F$, where $\nX = \nX(\pts)$.
\end{theorem}
\begin{proof}
  Correctness follows trivially from \cref{thm:bi-eval}.
  For complexity let $\dX = \deg_Xf$, $\dY = \deg_Yf$ and $\nY = \nY(\pts)$.
  Since $f$ is in standard form we have that $\dY < a$.
  Furthermore, since $f \in \riroch(m P_{\infty})$ we know that $a \dX + b \dY \leq m$. It follows that $\dX \dY < \dX a \leq m$.
  Since for each value of $X$ in $\X(\pts)$, there can be at most $a$ solutions in $Y$ to the \Cab curve equation, we know $\nY < a$.
  It follows from \cref{thm:bi-eval} that the cost of \cref{line:enc:1} is
  \begin{align*}
    &\Oh(d_Y\Multlog{d_X + \nX } + \nX \Multlog{d_Y + \nY})\\
    &\subset \Oh(\Multlog{m + a\nX }) \ .
  \end{align*}
\end{proof}

As can be seen, the complexity of this algorithm depends on parameters of the \Cab curve compared to the code length as well as the layout of the evaluation points: more specifically on how $an_X$ compares with the code length $n$.
In \cref{ssec:fast-unencoding} we will revisit the complexity for codes over \Cab curves that lie on semi-grids as well as \Cab curves which have many points.
In the worst case, the following corollary bounds the complexity in terms of the length of the code under very mild assumptions on the \Cab curve.
Note that this cost is still much better than encoding using a matrix-vector product in $\Oh(n^2)$ time.

\begin{corollary}
  \label{cor:enc_any_cab_code}
  In the context of \cref{algo:Enc}, let $q$ be the cardinality of $\F$ and assume $n \geq q$.
  Assume further that the genus $g$ of the \Cab curve satisfies $g \leq n$.
  Then the complexity of \cref{algo:Enc} is $\sOh(q\sqrt n) \subset \sOh(n^{3/2})$.
\end{corollary}
\begin{proof}
  There can at most be $q$ different $X$-coordinates in $\pts$, so $n_X \leq q$.
  Assuming w.l.o.g that $a < b$ we get $n \geq g = \tfrac 1 2 (a-1)(b-1) \geq \tfrac 1 2 (a-1)^2$, and hence $a \leq \sqrt{2n} + 1 \in \Oh(\sqrt n)$.
  Lastly, $m \leq n + 2g - 1 \in \Oh(n)$.
  The result follows from \cref{thm:Enc}.
\end{proof}

\section{A fast unencoding algorithm}
\label{sec:fast-unencoding}

We now consider the problem of \emph{unencoding}: we are given a codeword $\vec c \in \agcode_{\Hpol}(\pts,m)$ and we wish to find the message $\vec m = \phi^{-1}(\vec c) \in \F^k$, where $\phi$ is the encoding map defined in \cref{ssec:encoding-map}.
Following the discussion there, we factor $\phi$ as $\phi = \ev_{\pts} \circ \varphi$, where $\varphi: \F^k \mapsto \riroch(m P_\infty)$ which is a linear map that sends unit vectors to monomials in
\[
  \hat B \subseteq B := \{ x^i y^j \mid \deg_{a,b}(x^iy^j) \leq m \land j \leq a-1 \} \ .
\]
Recall that the evaluation map is injective on $\riroch(mP_\infty)$ whenever $m < n$, so in this case $\hat B = B$.
Otherwise, $\hat B$ is chosen such that if $g \in \ker(\ev_\pts)$, then $\LM_{\orderab}(g) \notin \hat B$, where $\orderab$ is the $(a,b)$-weighted degree monomial order, breaking ties by $x^b \orderab y^a$.

In \cref{ssec:fast-interp} we will first rephrase the above as a general interpolation problem for bivariate polynomials, and not use the fact that the $\pts$ are a subset of a low-degree \Cab curve.
In \cref{ssec:unencoding} we specialise and analyse the complexity of using this approach for the unencoding problem.

\subsection{Interpolation of Bivariate Polynomials}
\label{ssec:fast-interp}

In the following subsections, we suppose that we are given a set of $n$ points $\pts \subseteq \F^2$, not necessarily lying on a \Cab curve, and a corresponding collection of values $\vals \in \F^{\pts}$.
The \emph{interpolation problem} consists of finding a polynomial $f \in \F[X,Y]$ such that $f(\alpha, \beta) = \vals[\alpha, \beta]$ for all $(\alpha,\beta) \in \pts$.
Since there are many such polynomials, one usually imposes constraints on the set of monomials $X^i Y^j$ that may appear with non-zero coefficient in $f$.
For us, the relevant monomial support is the set $\hat B$ described above.

Depending on $\pts$ and the allowed monomial support, not all interpolation conditions can be satisified, but whenever they can, the solution could be found in time $O(n^2)$ using linear algebra, by precomputing a basis for the inverse of the evaluation map.

This section details a faster approach.
In \cref{ssec:fast-interp} we first use an efficient recursive algorithm to find an $\hat{f} \in \F[X,Y]$ such that $\hat{f}(x(P_i),y(P_i)) = c_i$ for $i = 1,\hdots,n$ while $\deg_Y \hat f < a$.
In general, $\hat f(x,y)$ is not the sought $f$ since its monomial support will be too large.
However, $f \in \hat f + \VanIdeal$, where $\VanIdeal$ is the ideal of polynomials in $\F[X,Y]$ vanishing at the points $P_1,\ldots,P_n$.
By the choice of $\hat B$, we will see that we can recover $f$ by reducing $\hat f$ modulo a suitable \Grobner basis $G$ of $\VanIdeal$ using a fast multivariate division algorithm.

\subsubsection{Finding a structured interpolation polynomial}
\label{sec:find-small-interp}

We will use the following explicit equation which finds an interpolating polynomial $\hat f \in \F[X,Y]$ with $\deg_Y \hat f < \nY(\pts)$ and $\deg_X \hat f < \nX(\pts)$.
Our algorithm for efficiently computing this polynomial essentially generalizes Pan's interpolation algorithm \cite{pan_simple_1994}, which was designed to work for points on grids, to work for points on semi-grids.

\begin{lemma}
  \label{lem:explicit_interpolation}
  Given a point set $\pts \in \F^2$ and interpolation values $\vals \in \F^\pts$, then $\hat f \in \F[X,Y]$ given by
  \begin{align}
    \label{eqn:explicit_interpolation}
    \hat{f} = \sum_{\alpha \in \X} \prod_{\alpha' \in \X \setminus \set{\alpha}} \frac{X-\alpha'}{\alpha - \alpha'}\sum_{\beta \in \Y_{\alpha}}\vals[\alpha,\beta]\prod_{\beta' \in \Y_{\alpha} \setminus \set{\beta}}\frac{Y-\beta'}{\beta - \beta'} \ ,
  \end{align}
  satisfies $f(\alpha,\beta) = \vals[\alpha,\beta]$ for all $(\alpha,\beta) \in \pts$.
\end{lemma}
\begin{proof}
  Let $(\alpha,\beta) \in \pts$. Then the only nonzero term in the first sum is the one corresponding to $\alpha$, while the only nonzero term in the second sum is the one corresponding to $\beta$, thus
  
 \[
   \hat{f}(\alpha,\beta) = \prod_{\alpha' \in \X \setminus \set{\alpha}}\frac{\alpha - \alpha'}{\alpha - \alpha'} \vals[\alpha,\beta] \prod_{\beta' \in \Y_{\alpha} \setminus \set{\beta}} \frac{\beta - \beta'}{\beta - \beta'} = \vals[\alpha,\beta] \ .
 \]
\end{proof}

Our strategy to compute $\hat{f}$ in an efficient manner can be viewed in the following way: we start by computing the polynomials $\hat{f}_{\alpha} := \hat{f}(\alpha,Y) \in \F[Y]$ for every $\alpha \in \X$ using univariate interpolation.
We then reinterpret our interpolation problem as being univariate over $(\F[Y])[X]$, i.e.~we seek $\hat f(X)$ having coefficients in $\F[Y]$ and such that $\hat f(\alpha) = \hat f_\alpha$.
However, to be more clear, and for a slightly better complexity (on the level of logarithms), we make this latter interpolation explicit.

Before we put these steps together to compute $\hat f$ in \cref{algo:BiInterp}, we therefore first consider the following sub-problem:
Given any subset $\Cset \subseteq \F$ and a table of univariate polynomials $\mathcal V \in \F[Y]^{\Cset}$ indexed by $\Cset$, compute the following bivariate polynomial:
\begin{equation}
  \label{eqn:combine_poly}
  h(X,Y) = \sum_{\alpha \in \Cset} \mathcal V[\alpha] \prod_{\alpha' \in \Cset\setminus \set{\alpha}}(X - \alpha') \in \F[X,Y] \ .
\end{equation}
For the $(\F[Y])[X]$ interpolation, we will follow an approach of univariate interpolation closely mimicking that of~\cite[Chapter 10.2]{von_zur_gathen_modern_2012}.
Firstly, we arrange the $X$-coordinates of the interpolation points in a balanced tree:
\begin{definition}
  Let $\Cset \subseteq \F$.
  A balanced partition tree of $\Cset$ is a binary tree which has subsets of $\Cset$ as nodes, and satisfies the following:
  \begin{enumerate}
    \item $\Cset$ is the root node.
    \item A leaf node is a singleton set $\{ \alpha \} \subseteq \Cset$.
    \item An internal node $\mathcal{N}$ is the disjoint union of its two children $\mathcal{N}_1,\mathcal{N}_2$ and they satisfy $\abs{\abs{\mathcal{N}_1} - \abs{{\mathcal{N}_2}}} \leq 1$.
  \end{enumerate}
  If $\tree$ is a balanced partition tree and $\mathcal{N} \subseteq \Cset$, we will write $\mathcal{N} \in \tree$ if $\mathcal{N}$ is a node of $\tree$, and denote by $\tree[\mathcal{N}]$ the set of its two child nodes.
\end{definition}

\begin{lemma}[\protect{Lemma 10.4}{ \cite{von_zur_gathen_modern_2012}}] \label{lemma:TreeVanish}
  There exists an algorithm $\TreeVanish_X$ which inputs a balanced partition tree $\tree$ of some $\Cset \subseteq \F$ and outputs the lookup table
  \begin{align*}
    \TreeVanish_X(\tree) := \big(\prod_{\alpha \in \mathcal{N}}(X - \alpha)\big)_{\mathcal{N} \in \tree} \in \F[X]^{\tree} .
\end{align*}
  The algorithm uses at most $\Oh(\Mult{k}\log(k)) \subset \sOh(k)$ operations in $\F$, where $k = |\Cset|$.
\end{lemma}

\begin{algorithm}
  \caption{$\LinMod$: $\F[Y]$-linear combination of vanishing polynomials} \label{algo:Combine}
  \KwInput{Points $\Cset \subseteq \F$, non-empty. \;
    Lookup table $\mathcal{V} \in \F[Y]^{\Cset}$.\;
    A balanced partition tree $\tree$ with $\Cset \in \tree$.\;
    $\VanTree = \TreeVanish_X(\tree) \in \F[X]^{\tree}$.\;
  }
  \KwOutput{$h(X,Y) \in \F[X,Y]$ as in \eqref{eqn:combine_poly}.
  }
  \If{$S = \set{\alpha}$}
  {
    \Return $\mathcal{V}[\alpha] \in \F[Y]$
  }
  \Else{
    $\set{\Cset_1,\Cset_2} \assign \tree[\Cset]$ \;
    \For{$k = 1,2$}{
      $\mathcal{V}_k \assign (\mathcal{V}[\alpha])_{\alpha \in \Cset_k} \in \F[Y]^{\Cset_k}$ \;
      $\hat{f}_k \assign \LinMod(\Cset_k, \mathcal{V}_k,\tree, \mathcal{U}) \in \F[X,Y]$
    }
    \Return $\hat{f}_1\VanTree[\Cset_2] + \hat{f}_2\VanTree[\Cset_1] \in \F[X,Y]$
  }
\end{algorithm}

With these tools in hand, \cref{algo:Combine} is an algorithm for computing $h(X,Y)$ as in \eqref{eqn:combine_poly}.
\begin{theorem}
  \label{thm:Combine}
  \cref{algo:Combine} is correct.
  If $\deg_Y \mathcal{V[\alpha]} < d$ for all $\alpha \in \Cset$, then the algorithm has complexity $\Oh(d \Multlog{k}) \subset \sOh(dk)$ operations in $\F$, where $k = \abs{\Cset}$.
\end{theorem}
\begin{proof}
  We prove correctness by induction on $\abs{\Cset}$.
  The base case of $\Cset = \{ \alpha \}$ is trivial.
  For the induction step, the algorithm proceeds into the else-branch and we note that $|\Cset_1|, |\Cset_2| < |\Cset|$ and so the induction hypothesis applies to $\hat f_k$ for $k = 1,2$:
  \[
    \hat f_k
    = \sum_{\alpha \in \Cset_k} \mathcal{V}_k[\alpha] \prod_{\alpha' \in \Cset_k \setminus \set{\alpha}}(X-\alpha') \ .
  \]
  Note that $\mathcal V_k[\alpha] = \mathcal V[\alpha]$ and $\Cset_1 \cup \Cset_2 = \Cset$, so we conclude that the polynomial returned by the algorithm is
  \begin{align*}
    &\hat{f}_1\VanTree[S_2] + \hat{f}_2\VanTree[S_1]  \\
    &= \sum_{\alpha \in \Cset_1} \mathcal{V}[\alpha] \prod_{\alpha' \in \Cset \setminus \set{\alpha}}(X-\alpha') + \sum_{\alpha \in \Cset_2} \mathcal{V}[\alpha] \prod_{\alpha' \in \Cset \setminus \set{\alpha}}(X-\alpha') \\
    &= \sum_{\alpha \in \Cset} \mathcal{V}[\alpha] \prod_{\alpha' \in \Cset \setminus \set{\alpha}}(X-\alpha') \ .
  \end{align*}
  For complexity let $T(k)$ denote the cost of \LinMod for $\abs{\Cset} = k$.
  At a recursive step we solve two subproblems of size at most $\lceil k/2 \rceil$; and we compute the expression $\hat{f}_1 \VanTree[\Cset_2] + \hat{f}_2 \VanTree[\Cset_1]$.
  Since $\VanTree[\Cset_k] \in \F[X]$ then each of the two products can be carried out using $\deg_Y \hat f_k + 1 \leq d$ multiplications in $\F[X]_{\leq k}$ and hence cost $\Oh(d\Mult{k})$ each.
  The addition $\hat{f}_1 \VanTree[\Cset_2] + \hat{f}_2 \VanTree[\Cset_1]$ costs a further $\Oh(d k)$.
  In total, we get the following recurrence relation:
  \begin{align*}
    T(k) &= 2T(k/2) + \Oh(d\Mult{k}) \ .
  \end{align*}
  This has the solution $T(k) = \Oh(d \Multlog{k}) + \Oh(k)T(1)$. $T(1)$ is the base case of the algorithm, and costs $\Oh(d)$.
\end{proof}

We are now in position to assemble the steps outlined for computing the interpolation polynomial given by \cref{lem:explicit_interpolation} following the steps outlined above by supplying \cref{algo:Combine} with the correct input; this is described in \cref{algo:BiInterp}.

Note that \cref{algo:BiInterp} makes use of the subroutines $\UniInterp_Y$ and $\UniMPE$ that were introduced in \cref{prop:uni-mpe} and \cref{prop:uni-interp}.

\begin{algorithm}
  \caption{$\BiInterp$: Bivariate interpolation} \label{algo:BiInterp}
  \KwInput{Points $\pts \subseteq \F^2$, non-empty. Lookup table of interpolation values $\vals \in \F^{\pts}$.\;
}
\KwOutput{The polynomial $\hat f \in \F[X,Y]$ given by \cref{lem:explicit_interpolation}.}
$\X \assign \X(\pts) \subseteq \F$ \;
\ForEach{$\alpha \in \X$ \label{BiInterp:for}}
{
  $\Y_{\alpha} \assign \Y_{\alpha}(\pts) \subseteq \F$ \;
  $\vals_{\alpha} \assign (\vals[\alpha,\beta])_{\beta \in \Y_{\alpha}} \in \F^{\Y_{\alpha}}$\;
 $\hat{f}_{\alpha} \assign \UniInterp_Y(\Y_{\alpha}, \vals_{\alpha}) \in \F[Y]$ \label{BiInterp:UniInterp}\;
}
$\tree \assign $ a balanced partition tree of $\X$ \label{BiInterp:tree}\;
$\VanTree \assign \TreeVanish_X(\tree) \in \F[X]^{\tree}$ \label{BiInterp:vanish}\;
$g \assign $ formal derivative of $\VanTree[\X] \in \F[X]$ \label{BiInterp:derivative}\;
$\mathcal{R} \assign \UniMPE(g, \X) \in \F^{\X}$ \label{BiInterp:UniMPE}\;
$\mathcal{V} \assign (\hat{f}_{\alpha}/ \mathcal{R}[\alpha])_{\alpha \in \X} \in \F[Y]^{\X}$ \label{BiInterp:scaling}\;
\Return $\LinMod(\X, \mathcal{V}, \tree, \VanTree)$ \label{BiInterp:LinMod}
\end{algorithm}

\begin{theorem} \label{thm:BiInterp}
  \cref{algo:BiInterp} is correct.
  It has complexity
  \[
    \Oh(\nY \Multlog{\nX } + \nX  \Multlog{\nY}) \subseteq \sOh(\nX \nY)
  \]
  operations in $\F$, where $\nX = \nX(\pts)$ and $\nY = \nY(\pts)$.
\end{theorem}

\begin{proof}
  Denote by $\tilde f \in \F[X,Y]$ the polynomial returned by the algorithm and $\hat f$ the polynomial  of \cref{lem:explicit_interpolation}, and we wish to prove $\tilde f = \hat f$.
  Note first that for any $\alpha \in \X$
  \[
    \hat f(\alpha, Y) = \sum_{\beta \in \Y_{\alpha}}\vals[\alpha,\beta]\prod_{\beta' \in \Y_{\alpha} \setminus \set{\beta}}\frac{y-\beta'}{\beta - \beta'} = \hat f_\alpha \ ,
  \]
  where $\hat f_\alpha$ is as computed in \cref{BiInterp:UniInterp}.
  By the correctness of $\LinMod$
  \[
    \tilde{f} = \sum_{\alpha \in \X} \hat f_{\alpha}/ \mathcal{R}[\alpha] \prod_{\alpha' \in \X \setminus \set{\alpha}}(X - \alpha') \ .
  \]
  Since $g = \sum_{\alpha \in \X}\prod_{\alpha' \in \X \setminus \set{\alpha}} (X - \alpha')$ we have that
  \[
    \mathcal{R}[\alpha] = \prod_{\alpha' \in \X \setminus \set{\alpha}}(\alpha - \alpha') \in \F\ , \quad \alpha \in \X \ .
  \]
  It follows that $\tilde f = \hat{f}$.

  For complexity we observe that computing $\tree$ in \cref{BiInterp:tree} cost $\Oh(\nX \log(\nX))$, and $\VanTree$ in \cref{BiInterp:vanish} costs $\Oh(\Multlog{\nX})$ by \cref{lemma:TreeVanish}. Computing $g$ in \cref{BiInterp:derivative} costs $\Oh(\nX)$ and $\mathcal{R}$ in \cref{BiInterp:UniMPE} costs $\Oh(\Multlog{\nX})$ by \cref{prop:uni-mpe}. The total cost of computing all the $\hat f_{\alpha}$ for $\alpha \in \X$ in \cref{BiInterp:UniInterp} is $\Oh(\nX\Multlog{\nY})$ by \cref{prop:uni-interp}. Since $\mathcal R[\alpha] \in \F$, the division $\mathcal{V}$ in \cref{BiInterp:scaling} costs $\Oh(\nX\nY)$ operations, one for each of the $\deg \hat f_\alpha < \nY$ coefficients in each of the $\nX$ polynomials $\hat f_\alpha$.
  Finally \cref{BiInterp:LinMod} costs $\Oh(\nY \Multlog{\nX})$ by \cref{thm:Combine}. The total cost becomes as in the theorem.
\end{proof}

\begin{remark}
  \label{rem:interp_semi-grid}
  If $\pts$ is a semi-grid then \cref{algo:BiInterp} has quasi-linear complexity in the input size $|\pts|$.
  Furthermore, since the output polynomial $\hat f$ has $\deg_X \hat f < n_X$ and $\deg_Y \hat f < \nY$, then the
  $n_X \nY$-dimensional $\F$-vector space of polynomials satisfying such degree restrictions must be in bijection with the $n$-dimensional choice of values $\vals$.
  Hence, \cref{algo:BiInterp} returns the \emph{unique} polynomial $\hat f$ satisfying these degree bounds and which interpolate the values.
  This is in contrast to the case where $\pts$ is not a semi-grid, as we will discuss in the following section.
\end{remark}

\begin{remark}
  If $\pts$ is very far from being a semi-grid, the performance of the algorithm can sometimes be improved by a proper blocking-strategy in use of \LinMod.
  For example, suppose that $\X = \set{\alpha_1, \hdots, \alpha_k}$, and furthermore assume that $|\Y_{\alpha_1}| = \hdots = |\Y_{\alpha_{k-1}}| = 1$ while $|\Y_{\alpha_k}| = k$, so that $n = 2k-1 \in \Theta(k)$.
  The cost of \cref{algo:BiInterp} will therefore be $\sOh(k^2)$.
  However, in this case we can split the points $\pts$ into $\pts_1 := \setcomp{(\alpha,\beta) \in \pts}{\alpha \neq \alpha_k}$ and $\pts_2 := \setcomp{(\alpha,\beta) \in \pts}{\alpha = \alpha_k}$, and the interpolation values into $\vals_s := (\vals[\alpha, \beta] / \gamma_s)_{(\alpha,\beta) \in \pts_s} \in \F^{\pts_s}$ for $s = 1,2$, where $\gamma_s := \prod_{\alpha' \in \X(\pts) \setminus{\X(\pts_s)}}(\alpha - \alpha')$ have been precomputed. We can then compute $\hat{f}_s := \BiInterp(\pts_s,\vals_s)$ and obtain the desired interpolating polynomial from \cref{lem:explicit_interpolation} as follows:
  \[
    \hat{f} = \sum_{s=1,2}\hat{f_s}\prod_{\alpha \in \X(\pts) \setminus{\X(\pts_s)}}(X - \alpha_i).
  \]
  Not counting precomputation, the total cost of this approach becomes $\sOh(k) = \sOh(n)$.
\end{remark}

\subsubsection{Reducing the support}
\label{sec:find-small-interp}

We now continue to deal with the problem that we wish to find an interpolation polynomial with a constrained monomial support.
We do not deal with the completely general question but one relevant to unencoding \Cab codes: we are given positive integers $a, b \in \ZZ_{> 0}$ and seek an interpolation polynomial $f$ which is ``reduced'' in a specific way according to the monomial ordering $\orderab$.
We will use $\hat f$ from the preceding section as an initial approximation which we then transform into the sought polynomial.

In the following, we will use \Grobner bases and assume that the reader is familiar with the basic notions; see e.g.~\cite{cox_ideals_2007}.
We will exclusively use the monomial order $\orderab$, so to lighten notation, in this section we write $\LM$ instead of $\LM_{\orderab}$.
Given a point set $\pts \in \F^2$, we will call a monomial $X^i Y^j$ ``reducible'' if there is a polynomial $h \in \F[X,Y]$ with $\LM(h) = X^i Y^j$ and such that $h(P) = 0$ for all $P \in \pts$.
In other words, we can ``remove'' $X^i Y^j$ from the support of $f$ by properly choosing $\alpha \in \F$ and setting $f := f - \alpha h$, and this does not increase the leading monomial of $f$.

We would now like to address the following problem:
\begin{problem}
  \label{prob:reduce_interp}
  Given a point set $\pts \subseteq \F^2$ and $a, b \in \ZZ_{\geq 0}$, and a polynomial $\hat{f} \in \F[X,Y]$, compute a polynomial $f \in \F[X,Y]$ satisfying
  \[
    f(P) = \hat f(P) \for P \in \pts \ ,
  \]
  and such that no monomial in the support of $f$ is reducible.
\end{problem}

Note that the polynomial $f - \hat f$ vanishes at all points $\pts$.
It is therefore natural to investigate the ideal
\begin{equation}
  \label{eqn:vanideal}
  \VanIdeal = \{ h \mid h(P) = 0 \for P \in \pts \} \subseteq \F[X,Y] \ .
\end{equation}
In fact we have the following lemma, which follows immediately from standard properties of \Grobner bases, see e.g.~\cite[Theorem 3.3]{cox_ideals_2007}:
\begin{lemma}
  \label{lem:grobner_reduces}
  In the context of \cref{prob:reduce_interp}, let $G$ be a \Grobner basis of $\VanIdeal$ given in \eqref{eqn:vanideal} according to $\orderab$.
  Then $f$ is the unique remainder of dividing $\hat f$ with $G$ using the multivariate division algorithm, i.e.~$f = \hat f \rem G$.
\end{lemma}

Before we give an efficient algorithm for solving \cref{prob:reduce_interp}, the following lemma recalls some simple structure on weighted-degree \Grobner bases:

\begin{lemma}
  \label{lem:minimal_interpolation}
  Let $\pts \subset \F^2$ and $a, b \in \ZZ_{\geq 0}$, let $\VanIdeal$ be given by \eqref{eqn:vanideal}, and let $G = \{ G_1,\ldots,G_t \}$ be a minimal \Grobner basis of $\VanIdeal$.
  Then $\deg_Y(\LM(G_i)) \neq \deg_Y(\LM(G_j))$ for $i \neq j$.
  In particular, if $\VanIdeal$ contains an element $g$ such that $\LM(g)=Y^A$ for some $A \in \mathbb{Z}_{\ge 1}$, then $|G| = t \leq A+1$.
\end{lemma}
\begin{proof}
  Assume oppositely that $\deg_Y(\LM(G_i)) = \deg_Y(\LM(G_j))$ for $i\neq j$, then either $\LM(G_i) \mid \LM(G_j)$ or $\LM(G_j) \mid \LM(G_i)$.
  Both of these contradict that $G$ is minimal.

  Assume now w.l.o.g.~that $\deg_Y(\LM(G_1) < \ldots < \deg_Y(\LM(G_t))$.
  Since $g \in \VanIdeal$ then $G$ contains an element, say $G_i$, whose leading monomial divides $Y^A$, i.e.~$\LM(G_i)=Y^k$ for $k \le A$.
  But then $\LM(G_i) \mid \LM(G_j)$ for $j > i$, so by the minimality of $G$ we must have $i = t$.
  Hence $t \leq k+1 \leq A+1$.
\end{proof}

We are now ready to solve \cref{prob:reduce_interp} in the generality that we need.
Our approach is to compute $\hat f \rem G$ using the fast multivariate division algorithm of van der Hoeven \cite{van_der_hoeven_complexity_2015}.

\begin{proposition}
  \label{prop:minimal_interpolation}
  Let $\pts \subset \F^2$ and $a, b \in \ZZ_{\geq 0}$, let $\VanIdeal$ be given by \eqref{eqn:vanideal}, and assume there is a $g \in \VanIdeal$ such that $\LM(g) = Y^A$ for $A \leq a$.
  Then there is an algorithm $\Reduce$ which inputs a polynomial $\hat f \in \F[X,Y]$ with $\deg_X(\hat f) < \nX := \nX(\pts)$ and $\deg_Y(\hat f) < A$, as well as the reduced \Grobner basis $G \subset \F[X,Y]$  of $\VanIdeal$ from \eqref{eqn:vanideal} according to $\orderab$, and which outputs a solution $f \in \F[X,Y]$ to \cref{prob:reduce_interp}.
  The complexity of the algorithm is:
  \[
    \Oh(a\Multlog{a\nX}\log(a\nX + ab)) \subset \sOh(a^2\nX) \ ,
  \]
\end{proposition}
\begin{proof}
  To achieve our target cost we will order $G$ as a list in specific way as explained below.
  The algorithm of \cite{van_der_hoeven_complexity_2015} can be seen as a fast way to carry out the following deterministic variant of classical multivariate division: we first initialize a remainder $R_1 = \hat f$.
  In iteration $k$ we set $i_k$ to be the smallest index such that $\LM(G_{i_k})$ divides \emph{some} term of the current remainder $R_k$.
  We then set $\mu_k = m_k/\LM(G_{i_k})$, where $m_k$ is the \emph{maximal} such term of $R_k$, according to $\orderab$.
  We then update $R_{k+1} = R_k - \mu_k G_{i_k}$.
  In the first iteration $k$ where no such $i_{k}$ can be found, we set $f := R_k$.
  The output of the division algorithm is $f, Q_1,\ldots,Q_t \in \F[X,Y]$ such that:
  \[
    \hat f = f + Q_1 G_1 + \ldots + Q_tG_t \ ,
  \]
  where $Q_i = \sum_{k, i_k = i} \mu_k$.
  The algorithm ensures that no term of $f$ is divisible by $\LM(G_i)$ for any $i$.

  The cost of the algorithm can be bounded as:
  \begin{align}
    \label{eq:hoeven-cost}
    \textstyle
    \Oh\big(\sum_{i=0}^t \Multlog{r_is_i}\log(ar_i+bs_i) \big) \ .
  \end{align}
  Here $r_i$ and $s_i$ denotes an a priori bound on $\deg_X(Q_i G_i)$ respectively $\deg_Y(Q_i G_i)$ for $i = 1,\ldots,t$, and $r_0$ and $s_0$ bounds $\deg_X(f)$ respectively $\deg_Y(f)$.

  We now specify how to order the $G_1,\dots,G_t$ to guarantee that the $r_i$ and $s_i$ are sufficiently small.
  Note that $\VanIdeal$ contains $\prod_{\alpha \in \X(\pts)}(X - \alpha)$, so $G$ must have an element whose leading monomial divides $X^{\nX}$: we set $G_1$ to be this element, i.e.~we have $\deg_X G_1 \leq n_X$ and $\deg_Y G_1 = 0$.
  Similarly, $G$ must contain an element whose leading monomial divides $Y^A$, the leading monomial of $g$ assumed by the proposition.
  We set this as $G_2$, and so $\LM(G_2) = Y^{d}$, where $d \leq A \leq a$, which also implies $\deg_Y(G_2) = d$.
  The order of the remaining elements of $G$ can be set arbitrary, but note that since $G$ is reduced then $\deg_X G_j < \nX$ and $\deg_Y G_j < d \leq a$ for $j = 3,\dots,t$.

  The degrees of $R_{k+1}$ compared to $R_k$ for any iteration $k$ will satisfy
  \begin{align*}
    \deg_X R_{k+1} &\leq \max( \deg_X R_k,\ \deg_X(\mu_kG_{i_k})) \\
    \deg_Y R_{k+1} &\leq \max( \deg_Y R_k,\ \deg_Y(\mu_kG_{i_k})) \ .
  \end{align*}
  We also know that $\deg_X\mu_k \leq \deg_X R_k$ and $\deg_Y\mu_k \leq \deg_Y R_k$.
  From this we can further analyze the degrees of $\mu_kG_{i_k}$ depending on $i_k$:
  \begin{itemize}
    \item $i_k = 1$, i.e.~$\deg_X R_k \geq \nX$.
    Then $\deg_X(\mu_k G_1) = \deg_X R_k$ and $\deg_Y(\mu_kG_1) \leq \deg_Y(R_k)$  since $\deg_Y(G_1) = 0$.
    \item $i_k = 2$, i.e.~$\deg_X R_k < \nX$ and $\deg_Y R_k \geq d$.
    Then $\deg_Y(\mu_k G_2) \leq \deg_Y R_k$.
    Also $\deg_X \mu_k \leq \deg_X R_k$ which means $\deg_X(\mu_k G_2) < 2\nX$.
    \item $i_k > 2$, i.e.~$\deg_X R_k < \nX$ and $\deg_Y R_k < d$.
    Then $\deg_X(\mu_k G_{i_k}) < 2\nX$ and $\deg_Y(\mu_k G_{i_k}) < 2d$.
  \end{itemize}
  Since we initialize by $R_1 = \hat f$ satisfying $\deg_X R_1 \leq \nX$ and $\deg_Y R_1 < a$, the above observations inductively ensure that $\deg_X R_k < 2\nX$ and $\deg_Y R_k < 2d \leq 2a$ for all iterations $k$.
  Hence for each $i = 1,\dots,t$ we get:
  \[
    \deg_X (Q_iG_i) < 2\nX \text{ and }\deg_Y (Q_iG_i) < 2a \ .
  \]
  It follows that $r_1,\dots,r_t \in \Oh(\nX)$ and $s_1,\dots,s_t \in \Oh(a)$, and since $f = R_k$ for the last iteration $k$, we also have $r_0 \in \Oh(\nX)$ and $s_0 \in \Oh(a)$.
  Combining this with \cref{lem:minimal_interpolation} we get the desired complexity estimate from \eqref{eq:hoeven-cost}.
\end{proof}

We will consider the computation of the reduced \Grobner basis $G \subset \F[X,Y]$ as precomputation, but a small discussion on the complexity of this computation is in order.
The ideal $\VanIdeal$ is well-studied, and the structure of a lex-ordered \Grobner basis with $x \prec y$ was already investigated by Lazard \cite{lazard_ideal_1985}.
Later and more explicitly, $\VanIdeal$ appeared as a special case of the ideals studied in soft-decoding of Reed--Solomon codes using the K\"otter--Vardy decoding algorithm, see e.g.~\cite{lee_interpolation_2006}: they give an algorithm to semi-explicitly produce a lex-ordered \Grobner basis, which they then reduce into a weighted-degree basis using essentially the Mulders--Storjohann algorithm \cite{mulders_lattice_2003}.
The total complexity seems to be $\Oh(a^3n^2)$.
\cite{jeannerod_fast_2016} gives a very general algorithm which can be applied to this case with better complexity, but there are a few technical details to be filled in: their algorithm works on $\F[X]$-modules and produces ``Popov bases'' of these.
It is folklore that this can yield a \Grobner basis algorithm for $\F[X,Y]$ ideals, if one has a bound on the maximal $Y$-degree appearing in a minimal \Grobner basis.
In our case that would be $a$.
Another technical detail is that both \cite{lee_interpolation_2006,jeannerod_computing_2017} deal with monomial orders of the form $\preceq_{1,c}$ for an integer $c$, and do not support $\preceq_{a,b}$ (\cite{jeannerod_computing_2017} supports other orders which do not translate into the form $\preceq_{u,v}$).
The order $\orderab$ is equivalent to the order $\preceq_{1,b/a}$, but $b/a \in \QQ \setminus \ZZ$.
Such rational weights are handled for a similar \Grobner basis computation in \cite{nielsen_sub-quadratic_2015}, and it seems reasonably that the approach could be combined with  either \cite{lee_interpolation_2006,jeannerod_computing_2017}, though the details are beyond the scope of this paper.

A different, generic approach is to observe that $\VanIdeal$ is a zero-dimensional ideal which means that we can use the FGLM algorithm to transform a \Grobner basis from the lex-order to one for $\orderab$ \cite{faugere_efficient_1993}.
The $X$-degree of the lex-order \Grobner basis output by $\VanIdeal$ will be $\nX$, and then the FGLM algorithm has running time at most $\Oh(a^3\nX^3)$.

\subsection{Fast unencoding}
\label{ssec:unencoding}
We will now apply the results from \cref{ssec:fast-interp} to the unencoding problem.
For a given codeword we compute the interpolating polynomial in \cref{lem:explicit_interpolation}, which we then reduce using techniques described in \cref{sec:find-small-interp}.
For the algorithm listing see \cref{algo:UnEnc}.

\begin{algorithm}
  \caption{$\UnEnc$: Unencoding of \Cab codes}\label{algo:UnEnc}
  \KwInput{
    A \Cab code $\agcode_\Hpol(\pts,m)$ with $\pts = \{ P_1, \hdots, P_n \}$ being finite rational places.
    A reduced \Grobner basis $G \subseteq \F[X,Y]$ of $\VanIdeal$ as defined in \eqref{eqn:vanideal} under monomial order $\orderab$, where $a = \deg_Y(\Hpol)$ and $b = \deg_X(\Hpol)$.
  }
  \KwOutput{$f_\Hpol \in \riroch(m P_{\infty})$ in standard form and monomial support in $\hat B$ given in \cref{ssec:encoding-map}, and s.t.~$f(P_i) = c_i \for i = 1, \hdots, n$.}
  $\vals \assign (c_i)_{P_i \in \pts} \in \F^{\pts}$ \;
  $\hat{f} \assign \BiInterp(\pts, \vals) \in \F[X,Y]$ \label{UnEnc:BiInterp}\;
  $f \assign \Reduce(\hat{f}, G) \in \F[X,Y]$ \; \label{UnEnc:Reduce}
  \Return $f_\Hpol \assign f(x,y) \in \F[x,y]$ \label{UnEnc:MapToFunF}
\end{algorithm}

\begin{theorem}
  \label{thm:unencoding}
  \cref{algo:UnEnc} is correct.
  It uses at most
  \[
    \Oh(a\Multlog{a\nX}\log(a\nX + ab))  \subset \sOh(a^2 \nX) \ ,
  \]
  operations in $\F$.
\end{theorem}
\begin{proof}
  By the correctness of \cref{algo:Combine} then $\hat f$ satisfies $\hat f(P_i) = c_i$ for $i=1,\ldots,n$ and by \cref{prop:minimal_interpolation} then so does $f$ and hence $f_\Hpol$.
  For the monomial support on $f_\Hpol$, note that $\Hpol \in \VanIdeal$ since $\Hpol$ vanishes at all of $\pts$.
  Since $\LM_{\orderab}(\Hpol) = Y^a$ then $G$ contains an element $G_1 \in \F[X,Y]$ with $\LM_{\orderab}(G_1) \mid Y^a$.
  Hence $\deg_Y(f) < \deg_Y(G_1) \leq a$ and so $f_\Hpol$ is obtained in standard form from $f$ using the natural inclusion of $\F[x,y]$ in $\F[X,Y]$, and the monomial support of $f_\Hpol$ corresponds exactly to the monomial support of $f$.
  By \cref{prop:minimal_interpolation}, $f$ contains no reducible monomials, which means that the support of $f_\Hpol$ is in $\hat B$.

  For complexity, \cref{thm:BiInterp} implies that \cref{UnEnc:BiInterp} costs
  \begin{align*}
    \Oh(\nY \Multlog{\nX } + \nX  \Multlog{\nY}) \ .
  \end{align*}
  By \cref{lem:explicit_interpolation} then $\deg_X(\hat f) < \nX$.
  Note that $\nY \leq a$ since for any $X$-coordinate, there can be at most $a$ solutions in $Y$ to the \Cab curve equation.
  Since $\Hpol \in \VanIdeal$ then by \cref{prop:minimal_interpolation}, \cref{UnEnc:Reduce} costs:
  \begin{align*}
    \Oh(a\Multlog{a\nX}\log(a\nX + ab)) \ ,
  \end{align*}
  and this dominates the total complexity.
  As mentioned \cref{UnEnc:MapToFunF} costs no operations in $\F$.
\end{proof}

Similar to the situation in \cref{ssec:fast-encoding}, the complexity of \cref{algo:UnEnc} depends on the value of $a$ compared to the code length as well as the layout of the evaluation points $\pts$.
We will return to analyze special cases in the following section, but the following corollary bounds the complexity in the worst case under very mild assumptions on the \Cab code.
For codes with $n \approx q$, this cost is not asymptotically better than the naive unencoding using linear algebra, which has cost $\Oh(n^2)$ assuming some precomputation, but one can keep in mind that AG codes are mostly interesting for use in constructing codes which are markedly longer than the field size.

\begin{corollary}
  \label{cor:unenc_any_cab_code}
  In the context of \cref{algo:UnEnc}, let $q$ be the cardinality of $\F$ and assume $n \geq q$.
  Assume further that the genus $g$ of the \Cab curve satisfies $g \leq n$.
  Then the complexity of \cref{algo:UnEnc} is $\sOh(q n) \subset \sOh(n^2)$.
\end{corollary}
\begin{proof}
  Assuming w.l.o.g. that $a < b$, we use the same upper bound $a \leq \sqrt n$ as in the proof of \cref{cor:enc_any_cab_code}.
\end{proof}

\begin{remark}
  We discuss for which \Cab codes it could be faster to use the unencoding approach discussed in \cref{sec:relat-work-bivar} using structured system solving, which costs $\sOh(a^{\omega-1} n)$.
  When ignoring hidden constants and log-terms, we see that whenever
  \begin{equation}
    \label{eqn:unenc_sysfaster}
    a^{3-\omega} > n/n_X \ ,
  \end{equation}
  structured system solving should be faster than our approach.
  Assuming $a < b$ then the genus $g \geq \tfrac 1 2 (a-1)^2$ and \Cab codes of most interest to coding theory satisfy $g < n$.
  Asymptotically replacing $a-1$ by $a$, we conclude that for \eqref{eqn:unenc_sysfaster} to hold, then at least
  $(n/n_X)^{2/(3-\omega)} < 2n$.
  Now $n_X \leq q$, so this is only possible if $n < 2q^{2/(\omega-1)}$.
  Taking the best known value for $\omega \approx 2.37286$ \cite{le_gall_powers_2014}, we get $n < 2q^{1.46}$.
  However, in practice we use matrix multiplication algorithms with values of $\omega$ quite close to $3$.
  For instance Strassen's multiplication algorithm has exponent $\hat\omega \approx 2.81$ \cite{strassen_gaussian_1969}, which would give the condition $n < q^{1.1}$, which can be considered a quite short \Cab code.
\end{remark}

\section{Applications}
\label{sec:applications}
In this section, we will apply \cref{algo:bi-eval} to the encoding and unencoding for various AG codes coming from \Cab curves. As we will show, in many interesting cases we can encode and unencode faster than $\Oh(n^2),$ in some cases even in $\sOh(n).$

\subsection{Quasi-linear encoding and unencoding for \Cab curves on semi-grids}
\label{ssec:fast-unencoding}

We already observed in \cref{rem:mpe_semi-grid} that our multipoint evaluation algorithm has very good complexity when the evaluation points lie on a semi-grid.
In this section we therefore investigate certain \Cab curves having many points that lie on a semi-grid and the complexity of our algorithms for codes over such curves.
Specifically, we will be interested in the following types of \Cab codes:

\begin{definition}
  \label{def:cab_semigrid}
  A \Cab code $\agcode_\Hpol(\pts, m)$ is called \emph{maximal semi-grid} if $\pts$ is a semi-grid with $\nY(\pts) = a =: \deg_Y(H)$ and $m < n$.
\end{definition}

Note the condition $m < n$ means that the encoding map is injective on $\riroch(m P_\infty)$ so the complications discussed in \cref{ssec:encoding-map} do not apply.

\begin{proposition}
  \label{prop:cab_semigrid_enc}
  Let $\agcode_\Hpol(\pts, m)$ be a maximal semi-grid \Cab code of length $n$.
  Then encoding using \cref{algo:Enc} has complexity $\Oh(\Multlog{n}) \in \sOh(n)$.
\end{proposition}
\begin{proof}
  Since $\agcode_\Hpol(\pts, m)$ is maximal semi-grid then $a n_X(\pts) = \nY(\pts) n_X(\pts) = |\pts| = n$, and further $m < n$.
  Hence, the result follows from \cref{thm:Enc}.
\end{proof}

The cost of unencoding using \cref{algo:UnEnc} is dominated by the call to $\Reduce$ in \cref{UnEnc:Reduce}.
The following proposition shows that for maximal semi-grid \Cab code, this step can be omitted.
First a small lemma.

\begin{lemma}
  \label{lem:semigrid_vanideal}
  Let $\agcode_\Hpol(\pts, m) \subseteq \F^n$ be a maximal semi-grid \Cab code.
  Let $\VanIdeal_\Hpol = \{ g \in \F[x,y] \mid g(P) = 0 \textrm{ for all } P \in \pts \}$.
  Then $\VanIdeal_\Hpol = G \cdot \F[x,y]$, where $G = \prod_{\alpha \in \X(\pts)} (x - \alpha)$.
\end{lemma}
\begin{proof}
  Clearly $G \in \VanIdeal_\Hpol$ so $G \cdot \F[x,y] \subseteq \VanIdeal_\Hpol$.
  For the other inclusion, observe first that we could write $\VanIdeal_\Hpol = \bigcup_{i \in \ZZ}\riroch(-D + i P_{\infty})$, where $D = \sum_{P \in \pts} P$.
  Note now that the pole order of $x - \alpha$ at $P_\infty$ is $a$, and it has poles nowhere else.
  On the other hand, it has a zero at each of the points $\{ (\alpha, \beta) \mid \beta \in \Y_\alpha(\pts) \}$ and there are $\nY(\pts) = a$ of them.
  Hence, the divisor of $x - \alpha$ is given exactly as
  \[
    (x - \alpha) = \sum_{\beta \in \Y_{\alpha}}P_{(\alpha,\beta)} - aP_{\infty} \ ,
  \]
  where $P_{\alpha, \beta} \in \pts$ is the place corresponding to the point $(\alpha, \beta) \in \pts$.
  It follows that $(G) = D - n P_{\infty}$.
  Therefore, for any $z \in \VanIdeal_\Hpol$ we have $z/G \in \riroch(s P_\infty)$ for some $s \in \ZZ$, i.e.~$z/G$ has only poles at $P_\infty$ and so $z/G \in \F[x,y]$.
  Hence $\VanIdeal_\Hpol \subseteq G \cdot \F[x,y]$, and we conclude equality.
\end{proof}

\begin{proposition}\label{prop:fhat_is_f}
  Let $\agcode_\Hpol(\pts, m) \subseteq \F^n$ be a maximal semi-grid \Cab code and assume $m < n$.
  Let $f \in \riroch(m P_{\infty})$ and let $\hat f \in \F[X,Y]$ be given by \cref{lem:explicit_interpolation} where $\vals[P] = f(P)$ for all $P \in \pts$.
  Then $f = \hat f(x,y)$.
\end{proposition}
\begin{proof}
  We write $\nY$ instead of $\nY(\pts)$ in the following, and similarly for $\X, \Y_\alpha, n_X$, etc.
  Note first that $\hat f \in \F[X,Y]$ has $\deg_Y{\hat f} < \nY = a$ so $\hat f(x,y)$ is obtained in standard form by the natural inclusion of $\F[x,y]$ in $\F[X,Y]$.

  Let $G(x)$ and $\VanIdeal_\Hpol$ be given as in \cref{lem:semigrid_vanideal}.
  Note that $f - \hat f(x,y) \in \VanIdeal_\Hpol = G \cdot \F[x,y]$.
  Write $f - \hat f(x,y) = \tilde f_0(x) + \tilde f_1(x) y + \ldots + \tilde f_{a-1}(x) y^{a-1}$ in standard form with $\tilde f_i \in \F[x]$.
  Since $G(x)$ is univariate in $x$, then $G(x)$ must divide every $\tilde f_i(x)$.
  Hence, if $f - \hat f(x,y) \neq 0$, we must have $\deg_x G \leq \deg_x(\tilde f_i)$ for some $i \in \{ 0,\ldots,a-1\}$.
  But
  \begin{align*}
    \deg_x(\tilde f_i)
    &\leq \max(\deg_{a,b}(f)/a,\ \deg_X(\hat f)) \\
    & \leq \max(m/a,\ n_X-1) \\
    & < n_X = \deg_x G \ ,
  \end{align*}
  where the last inequality follows from $\pts$ being a semi-grid and so $m < n = |\pts| = \nY n_X = a n_X$.
  This is a contradiction.
\end{proof}

\begin{proposition}
  \label{prop:cab_semigrid_unenc}
  Let $\agcode_\Hpol(\pts, m)$ be a maximal semi-grid \Cab code of length $n$.
  There is an algorithm for unencoding $\agcode_\Hpol(\pts, m)$ with complexity $\Oh\big(\nY\Multlog{n_X} + n_X\Multlog{\nu_Y}\big) \in \sOh(n)$.
\end{proposition}
\begin{proof}
  The algorithm is simply \cref{algo:UnEnc} with the following two changes:
  \begin{enumerate}
    \item We return $\hat f$ in \cref{UnEnc:BiInterp} and skip \cref{UnEnc:Reduce}.
    \item We do not take the \Grobner basis $G \in \F[X,Y]$ as input.
  \end{enumerate}
  That this algorithm is correct follows from \cref{prop:fhat_is_f}, since the code is maximal semi-grid.
  The big-$\Oh$ complexity is exactly that of \cref{thm:BiInterp} and the relaxation follows from $n_X \nY = n$.
\end{proof}

We stress that in contrast to the general unencoding algorithm, \cref{algo:UnEnc}, the unencoding algorithm for maximal semi-grid codes requires no precomputation.
We proceed by showing that several interesting classes of \Cab curves admit long maximal semi-grid codes.

\subsubsection{The Hermitian curve}
\label{ssec:hermitian}

\newcommand{\HermH}{{\mathcal H}}
\newcommand{\Herm}{\Hpol_\HermH}
The Hermitian curve is defined for fields $\F = \F_{q^2}$ for some prime power $q$ using the defining polynomial
\begin{equation}
  \label{eqn:hermitian_curve}
  \Herm(X,Y) = Y^q + Y - X^{q+1} \ .
\end{equation}
It is easily checked that it is a \Cab curve, with $a=q$ and $b=q+1$.
We say that a \Cab code is a \emph{Hermitian code} if the curve equation that is used is \cref{eqn:hermitian_curve}.

The Hermitian curve and its function field are well known and have been studied extensively in the literature, see e.g. \cite[Lemma 6.4.4]{stichtenoth_algebraic_2009}.
For instance, using the trace and norm maps of the extension $\F_{q^2} : \F_{q}$, it is easy to show that for every $\alpha \in \F_{q^2}$ there exist precisely $q$ distinct elements $\beta \in \F_{q^2}$ such that $\Herm(\alpha, \beta) = 0$.
In other words, the set of rational points $\pts_{\mathcal H}$ on $\Herm$ form a semi-grid with $\X(\pts) = \F_{q^2}$ and $\nY(\pts) = q = a$, i.e.~a total of $q^3$ points.

Therefore, the Hermitian code $\agcode_{\Herm}(\pts_\HermH,m)$ with $m < n = q^3$ is maximal semi-grid since $\pts_\HermH$ is a semi-grid with $\nY(\pts) = \deg_Y(H_{\HermH})$.
We immediately get the following corollary of \cref{prop:cab_semigrid_enc} and \cref{prop:cab_semigrid_unenc}:
\begin{corollary}\label{cor:hermit-cost}
  Let $\agcode_{\Herm}(\pts_\HermH,m) \in \F_{q^2}^n$ be an Hermitian code with $m < n = |\pts_\HermH|$, and $\pts_\HermH$ the set of all rational points on the $\Herm$ as given in \eqref{eqn:hermitian_curve}.
  Then encoding with \cref{algo:bi-eval} uses
  \begin{align*}
    \Oh(\Multlog{m + n}) \subset \sOh(n)
  \end{align*}
  operations in $\F_{q^2}$.
  Unencoding with the algorithm of \cref{prop:cab_semigrid_unenc} uses
  \begin{align*}
    \Oh(q (\Mult{q^2}) \log(q)) \subset \sOh(n)
  \end{align*}
  operations in $\F_{q^2}$.
\end{corollary}

Note that encoding and unencoding also has quasi-linear complexity for any shorter Hermitian code, where we use any sub semi-grid of the points $\pts \subset \pts_\HermH$ as long as $\nY(\pts) = \nY(\pts_\HermH) = a$.
This corresponds to selecting a number $n_X \leq q^2$ of $X$-coordinates, and for each choosing all the $q$ points in $\pts_\HermH$ having this $X$-coordinate.

\subsubsection{Norm-trace and other Hermitian-like curves}

\newcommand{\NormN}{{\mathcal{N}}}
\newcommand{\NormNe}{{\NormN,e}}
\newcommand{\Norm}{\Hpol_\NormN}
\newcommand{\Norme}{\Hpol_\NormNe}

It is not hard to find other examples of \Cab curves which admit large maximal semi-grid codes.
In this subsection we give examples of curves from the literature which have this property.

Let $q$ be a prime power and $r \in \mathbb{Z}_{\ge 2}$. Further let $e$ be a positive integer dividing the integer $(q^r-1)/(q-1)$ and define
\begin{equation}
  \label{eqn:norme_curve}
  \Norme(X,Y) = X^{q^{r-1}}+\cdots+X^q + X - Y^{e} \in \F_{q^r}[X,Y] \ ,
\end{equation}
This is a \Cab polynomial with $a=e$ and $b=q^{r-1}$.

We first look at the case $e=(q^r-1)/(q-1)$ which gives rise to the norm-trace curves studied in \cite{geil_2003}.
For $r=2$ they simplify to the Hermitian curve.
Similarly to the Hermitian curve, one obtains that the set of points $\pts_\NormN$ of $\Norm(X,Y)$ form a semi-grid with $\nX(\pts_H) = q^r$ and $\nY(\pts_H) = a$ and therefore $|\pts_\NormN| = q^{2r-1}$.
Corollary \ref{cor:hermit-cost} generalizes directly and shows that one-point norm-trace codes can be encoded and unencoded in quasi-linear time:

\begin{corollary}\label{cor:norm_trace-cost}
  Let $q$ be a prime power, $r \in \mathbb{Z}_{\ge 2},$ and $e=(q^r-1)/(q-1).$
  Further let $\Norm(X,Y)$ be as in \eqref{eqn:norme_curve}, and $\pts_\NormN$ the set of rational points on $\Norm$.
  Let $\agcode_{\Norm}(\pts_\NormN, m)$ be the corresponding \Cab code for some $m < n = |\pts_\NormN| = q^{2r-1}$.
  Then encoding with \cref{algo:bi-eval} uses
  \begin{align*}
    \Oh(\Multlog{m + n}) \subset \sOh(n)
  \end{align*}
  operations in $\F_{q^r}$.
  Unencoding with the algorithm of \cref{prop:cab_semigrid_unenc} uses
  \begin{align*}
    \Oh(q^{r-1} (\Mult{q^r}) \log(q^r)) \subset \sOh(n)
  \end{align*}
  operations in $\F_{q^r}$.
\end{corollary}

If $e<(q^r-1)/(q-1),$ the equation $\Norme(\alpha,\beta)=0$ has $q^{r-1}+e(q^r-q^{r-1})$ solutions in $\F_{q^r}^2$.
The small term $q^{r-1}$ comes from the solutions where $\beta=0$ and $\alpha^{q^{r-1}}+\cdots+\alpha=0$.
The remaining $e(q^r-q^{r-1})$ points again form a semi-grid $\pts_\NormNe$ with $\nY(\pts_\NormNe) = e = a$ and $\nX =q^r-q^{r-1}$, and can therefore be used to construct long codes with efficient encoding and decoding.
A special case of these curves, where $r$ is even and $e$ divides $q^{r/2}+1$ was considered in \cite{castle_2009}.
We obtain the following:

\begin{corollary}\label{cor:hermitlike-cost}
  Let $q$ be a prime power, $r \in \mathbb{Z}_{\ge 2},$ and $e$ an integer dividing $(q^r-1)/(q-1),$ but not equal to it.
  Further let $\Norme$ be given by \eqref{eqn:norme_curve}, and $\pts_\NormNe$ the set of finite rational points on $\Norme$.
  Let $\agcode_{\Norme}(\pts_\NormNe, m)$ be the corresponding \Cab code for some $m < n = |\pts_\NormNe| = e(q^r-q^{r-1})$.
  Then encoding with \cref{algo:bi-eval} uses
  \begin{align*}
    \Oh(\Multlog{m + n}) \subset \sOh(n)
  \end{align*}
  operations in $\F_{q^r}$.
  Unencoding with the algorithm of \cref{prop:cab_semigrid_unenc} uses
  \begin{align*}
    \Oh(e (\Mult{q^r}) \log(q^r)) \subset \sOh(n)
  \end{align*}
  operations in $\F_{q^r}$.
\end{corollary}

\subsection{Fast encoding for good families of \Cab curves}
\label{ssec:encoding-good-curves}

We will now investigate the complexity of \cref{algo:bi-eval} for families of curves that have many points. More precisely, we will consider curves for which the number of points is asymptotically close to the \emph{Hasse--Weil bound}:

\begin{theorem}[Hasse--Weil bound, \protect{\cite[Theorem 5.2.3]{stichtenoth_algebraic_2009}}]
If $N$ is the number of rational places of an algebraic function field $\ffield$ over $\F_q$, then
\begin{align*}
  N \leq 2g\sqrt q + (q+1),
\end{align*}
where $g$ is the genus of $\ffield$.
\end{theorem}
For \Cab curves we know $g = \tfrac 1 2 (a-1)(b-1)$, so the Hasse--Weil bound upper-bounds the length $n$ of a \Cab code $\agcode(\Hpol, m)$ over $\F_q$ as follows:
\[
  n \leq \HW(\Hpol) := (a-1)(b-1)\sqrt q + q \ .
\]
Observe that this is one less than the bound on the number of rational places of the function field corresponding to the \Cab curve, since the rational place at infinity is not included as an evaluation point.
The Hermitian curve is an example of a curve which attains the Hasse--Weil bound.

\begin{lemma}
  \label{lem:good_curves_bounds}
  Let $\agcode_{\Hpol}(\pts, m)$ be a \Cab code over $\F_q$, with $a = \deg_Y(\Hpol)$, $b = \deg_X(\Hpol)$ and $a < b$.
  Let $n$ be the length of the code, and assume that $n \geq q$ as well as $n \geq c \cdot \HW(\Hpol)$ for some constant $c \in (0, 1]$.
  Then the following upper bounds hold:
  \begin{align*}
    a    &< \sqrt{\frac n {c\sqrt q}} + 1 \ ;\\
    qa   &< \frac{n^{5/4}}{\sqrt c} + n \ ;\\
    qa^2 &< \frac{n^{3/2}} c + \frac{2n^{5/4}}{\sqrt c} + 2n \ . \\
  \end{align*}
\end{lemma}
\begin{proof}
  Since $n \geq c \cdot \HW(\Hpol) \geq c(a-1)(b-1)\sqrt{q}$ and $a < b$ we have that
  \[
    n > c (a-1)^2 \sqrt q \iff \frac n {c \sqrt q} > (a-1)^2 \ ,
  \]
  which gives the first bound.
  Since $q \leq n$ we then also get
  \[
    qa < q^{3/4} \sqrt{\frac n c} + q \leq \frac {n^{5/4}}{c^{1/2}} + n \ .
  \]
  Lastly, $qa^2 < q\big( (a-1)^2 + 2a)$ and so
  \[
    qa^2 < \frac{n\sqrt q}{c} + 2qa \leq \frac{n^{3/2}}{c} + 2qa \ ,
  \]
  and the last bound follows by inserting our earlier bound for $qa$.
\end{proof}

In the following, we will discuss the asymptotic complexity of encoding and unencoding for infinite families of \Cab codes.
Note that for this to make any sense, the length of the codes must go to infinity and therefore the size of the fields over which the codes are defined must also go to infinity.
In the remainder of the section, when we introduce an infinite family of \Cab curves $\Gamma = \{ \agcode_{\Hpol_i}(\pts_i, m_i) \}_{i \in \ZZ_{\geq 1}}$, we also implicitly introduce the related variables: $q_i$ is the prime power such that $\Hpol_i$ and the code $\agcode_{\Hpol_i}(\pts_i, m_i)$ is defined over $\F_{q_i}$; $a_i := \deg_Y (\Hpol_i)$ and $b_i := \deg_X (\Hpol_i)$ and we assume $a_i < b_i$; and $n_i$ is the length of the code for each $i$.

For an infinite sequence of real numbers $\vec c = (c_1,c_2,\ldots) \in\ (0,1]^\infty$, we say that the code family $\Gamma$ is \emph{asymptotically $\vec c$-good} if $n_i \geq c_i \HW(\Hpol_i)$ for all $i = 1, 2, \ldots$.

\begin{theorem}
  \label{thm:hasse-weil-enc}
  Let $\Gamma = \{ \agcode_{\Hpol_i}(\pts_i, m_i) \}_{i \in \ZZ_{\geq 1}}$ be an infinite family of \Cab codes with related variables $q_i, a_i, b_i, n_i$, with $a_i < b_i$, which is asymptotically $\vec c$-good for $\vec c = (c_1,c_2,\ldots)$.
  Then the asymptotic complexity of encoding $\agcode_{\Hpol_i}(\pts_i, m_i)$ for $i \rightarrow \infty$ using \cref{algo:Enc} is
  \begin{align*}
    \Oh\big(\Multlog{m_i + n_i^{5/4}/\sqrt{c_i}}\big) \subset \sOh\big(m_i + n_i^{5/4}/\sqrt{c_i}\big)
  \end{align*}
  operations in $\F_{q_i}$.
  The asymptotic complexity of unencoding $\agcode(\Hpol_i, m_i)$ for $i \rightarrow \infty$ using \cref{algo:UnEnc} is
  \begin{align*}
    \Oh\big(\Mult{n_i^{3/2}/c_i}(\log(n_i^{3/2}/c_i))^2\big) \subset \sOh\big(n_i^{3/2}/c_i\big)
  \end{align*}
  operations in $\F_{q_i}$.
\end{theorem}

\begin{proof}
  \cref{thm:Enc} gives the asymptotic cost of encoding $\agcode(\Hpol_i,m_i)$ as
  \begin{align*}
    \Oh(\Multlog{m_i + a_i n_{X, i}}) \ ,
  \end{align*}
  operations in $\F_{q_i}$, where $n_{X, i}$ is the number of distinct $X$-coordinates in the evaluation points used in $\agcode_{\Hpol_i}(\pts_i, m_i)$.
  Since $n_{X,i} \leq q_i$ we can use the bound on $a_i q_i$ given by \cref{lem:good_curves_bounds}.
  In the big-Oh notation, the lower-order terms can be ignored, and this gives the estimate of encoding.

  For unencoding, the cost is given by \cref{thm:unencoding} as
  \begin{align*}
    \Oh(\Mult{a_i^2 n_{X, i}}\log(a_i n_{X, i})\log(a_i n_{X,i} + a_i b_i) \ .
  \end{align*}
  Since $n_{X,i} \geq c_i (a_i-1)(b_i-1) \sqrt q$, then $a_i b_i \in \Oh(n_{X,i}/c_i)$.
  We then use $a_i^2 n_{X,i} \leq a_i^2 q_i$ which is then bounded by \cref{lem:good_curves_bounds}.
  Note that since $c_i \in (0, 1]$ then $\sqrt{c_i} > c_i$ and so we always have $n_i^{3/2}/c_i \geq n_i^{5/4}/\sqrt{c_i}$, so we need only keep the term $n_i^{3/2}/c_i$ in the asymptotic estimate.
\end{proof}

Let us discuss some consequences of this result.
Consider first that all $c_i = c$ for some fixed constant $0 < c \leq 1$, i.e.~that all the curves of the codes in $\Gamma$ are a factor $c$ from Hasse--Weil: then we can encode using $\sOh(n_i^{5/4})$ operations in $\F_{q_i}$, or $\sOh(n_i^{5/4} \log(q_i))$ bit-operations, which is significantly better than the naive approach of roughly $\Oh(n_i m_i)$ operations in $\F_{q_i}$.
Though the constant $c$ disappears in the asymptotic estimate, \cref{thm:hasse-weil-enc} describes by the dependency on $1/\sqrt c$ how the encoding algorithm fares on asymptotically worse families compared to asymptotically better families.
For instance, if $c = 1/100$ and $\Gamma$ consists of \Cab codes over curves achieving only $1\%$ of the Hasse--Weil bound, the running time of the algorithm will be roughly $10$ times slower pr.~encoded symbol compared to running the algorithm on a family which attains the Hasse--Weil bound.

\cref{thm:hasse-weil-enc} is useful also for families of curves which get farther and farther away from the Hasse--Weil bound.
Indeed as long as $1/\sqrt{c_i}$ grows slower than $n_i^{3/4}$, i.e. $c_i$ stays above $n_i^{-9/16}$ times a constant, we still get an improvement over the naive encoding algorithm.

\begin{remark}
  An alternative unencoding approach of structured system solving described in \cref{sec:relat-work-bivar} has a cost of $\sOh(a^{\omega-1} n)$.
  It does not seem easy to completely fairly compare this cost with that of \cref{thm:hasse-weil-enc}, but we can apply a similar over-bounding strategy and get a single exponent for $n$:
  By \cref{lem:good_curves_bounds} then $a_i < n_i^{1/2} q_i^{-1/4} c^{-1/2}$.
  Note that $q_i^2 \geq n_i$ so we get $a_i < n_i^{3/8} c_i^{-1/2}$ and hence if we replace $a$ by this bound in the cost of the structured system solving we get $\sOh(n_i^{1 + 3/8(\omega-1)})$ for $i \rightarrow \infty$.
  This is roughly $\sOh(n_i^{1.515})$ if we use the best known value for $\omega \approx 2.37286$ \cite{le_gall_powers_2014}.
  For the more practical matrix multiplication algorithms of Strassen with $\hat \omega \approx 2.81$ \cite{strassen_gaussian_1969}, we get $\sOh(n_i^{1.68})$, and simply replacing $\omega$ by $3$ yields $\sOh(n_i^{1.75})$.
\end{remark}

\section*{Acknowledgments}
The authors are grateful to the referees for their suggestions on how to improve the paper. The authors would also like to acknowledge the support from The Danish Council for Independent Research (DFF-FNU) for the project \emph{Correcting on a Curve}, Grant No.~8021-00030B.

\bibliographystyle{abbrv}
\bibliography{bibtex}

\end{document}